\documentclass[12pt]{amsart}

\usepackage{amsgen,amsmath,amsxtra,amssymb,amsfonts,amsthm,color,amssymb}

\usepackage{tikz} 
\usetikzlibrary{arrows}
\numberwithin{equation}{section}
\newtheorem{example}{\sc Example}[section]
\newtheorem{dhef}[example]{\sc Definition}
\newtheorem{lemma}[example]{\sc Lemma}

\newtheorem{remark}[example]{\sc Remark}
\newtheorem{prop}[example]{\sc Proposition}
\newtheorem{theo}[example]{\sc Theorem}

\def\be{\begin{equation}}
\def\ee{\end{equation}}
\def\bc{\begin{cases}}
\def\ec{\end{cases}}

\newcommand{\na}{\mathbb{N}}

\newcommand{\re}{\mathbb{R}}
\newcommand{\rn}{\mathbb{R}^{N}}

\newcommand{\huz}{H^1_0 (\Omega)}
\newcommand{\elle}[1]{L^{#1}(\Omega)}

\def\al{\alpha}
\def\ga{\gamma}
\def\de{\delta}
\def\eps{\varepsilon}
\def\th{\theta}
\def\la{\lambda}
\def\La{\Lambda}
\def\si{\sigma}
\def\vp{\varphi}

\def\dys{\displaystyle}
\long\def\salta#1{\relax}

\def\rife#1{(\ref{#1})}
\def\io{\int_{\Omega}}
\def\norma#1#2{\|#1\|_{\lower 4pt \hbox{$\scriptstyle #2$}}}
\newcommand{\LL}{\>\hbox{\vrule width.2pt \vbox to 7pt{\vfill\hrule width 7pt height.2pt}}\>}
\def\mis{{\rm meas}}
\def\dive{{\rm div}}
\def\D{\nabla}
\def\ol#1{\overline{#1}}
\def\ul#1{\underline{#1}}
\def\pll{($P_{\ell}$)}
\def\vll{v_{\ell}}
\def\rll{R_{\ell}}

\def\bk{\color{black}}

\author[L. Boccardo]{Lucio BOCCARDO}
\address{Lucio Boccardo,
Dipartimento di Matematica,
Universit\`a di Roma ``La Sapienza'',
P.le Aldo Moro 2, 00185 Roma, Italy}
\email{boccardo@mat.uniroma1.it}
\author[T. Leonori]{Tommaso LEONORI}
\address{Tommaso Leonori,
Departamento de An\'alisis Matem\'atico,
Universidad de Granada,
Campus Fuentenueva s/n, 18071 Granada, Spain}
\email{leonori@ugr.es}
\author[L. Orsina]{Luigi ORSINA}
\address{Luigi Orsina,
Dipartimento di Matematica,
Universit\`a di Roma ``La Sapienza'',
P.le Aldo Moro 2, 00185 Roma, Italy}
\email{orsina@mat.uniroma1.it}
\author[F. Petitta]{Francesco PETITTA}
\address{Francesco Petitta,
Departamento de An\'alisis Matem\'atico,
Universitat de Valencia,
C/ Dr. Moliner 50, 46100 
Burjassot, Valencia, Spain}
\email{francesco.petitta@uv.es}

\title[\sc\tiny Singular quasilinear equations]
{\sc\normalsize Quasilinear elliptic equations with singular quadratic growth terms}

\begin{document}

\begin{abstract}
In this paper we deal with positive solutions for singular quasilinear problems whose model is
$$
\bc
-\Delta u + \frac{|\D u|^2}{(1-u)^\gamma}=g & \mbox{in $\Omega$,}\\
\hfill u=0 \hfill & \mbox{on $\partial\Omega$,}
\ec
$$
where $\Omega $ is a bounded open set of $\rn$, $g\geq 0 $ is a function in some Lebesgue space, and $\gamma>0$. 
We prove both existence and nonexistence of solutions depending on the value of $\gamma$ and on the size of $g$. 
\end{abstract}

\maketitle

\section{Introduction and statement of results}

Let $\Omega$ be an open bounded set of $\rn$ ($N\geq 1$). We are interested in the study of singular elliptic quasilinear problems of the type
\be\label{main}
\bc
-\dive (M(x)\D u) + h(u)|\D u|^2=g & \mbox{in $\Omega$,}\\
\hfill u=0 \hfill & \mbox{on $\partial\Omega$,}
\ec
\ee
where $M(x)=(m_{ij}(x))$, $i,j=1,\ldots,N$ is a symmetric matrix whose
coefficients $m_{ij}:\Omega \to \mathbb{R}$ are measurable functions such that:
\begin{equation} \label{a1}
\hfill \alpha |\xi|^2 \leq M(x)\, \xi \cdot \xi \quad \mbox{ and } \quad |M(x)|\leq
\beta\,,
\end{equation}
for almost every $x$ in $\Omega$, for every $\xi$ in $\rn$, where $0< \alpha \leq \beta$.
Moreover we suppose that $g \geq 0$ is a function in $\elle1$. 
As far as the lower order term is concerned, we assume that $h(s)$ is a continuous and strictly increasing function such that
\begin{equation}\label{h1}
h: [0,\si)\to \re^+\,,\quad \sigma>0\,,
\end{equation}
and that
\begin{equation}\label{h2}
\dys \lim_{s\to \si^- } h(s)=+\infty \,.
\end{equation}
To fix the ideas we will often refer to the following model problem 
\be\label{mod}
\bc
-\Delta u + \frac{|\D u|^2}{(1-u)^\gamma}=g & \mbox{in $\Omega$,}\\
\hfill u=0 \hfill & \mbox{on $\partial\Omega$,}
\ec
\ee
where $\gamma>0$, and $g$ belongs to $\elle1$, $g \geq 0$.

\medskip 

The peculiarity of this kind of equations is that the lower order term ``forces'' the solutions to be bounded. This phenomenon, due to the strongly regularizing effect of the lower order term, has already been observed for semilinear equations such as
$$
\bc
-\dive(M(x)\D u) + h(u) = g & \mbox{in $\Omega$,}\\
\hfill u=0 \hfill & \mbox{on $\partial\Omega$,}
\ec
$$
with $h$ as above. It has been proved that existence occurs even if $g\in \elle1$ (see \cite{boc}) and, more in general, if $g$ is a measure (see \cite{dupopo}). 

\medskip

Here we are interested in singular lower order terms depending on the gradient. We recall that existence of an $\huz$ solution for \rife{main} with nonsingular $h$ is nowadays well known: we refer, just to quote some, to the papers \cite{bbm}, \cite{bmps}, \cite{bg3}, \cite{bn} \cite{bgo2} and \cite{brmapo}, where several existence and nonexistence results are proved. 

On the other hand, \bk the study of existence and nonexistence for solution to Dirichlet problems as \rife{main} with a function $h$ which becomes singular at $s = 0$ has been recently carried out. We refer to \cite{a6}, \cite{b-me}, \cite{ablp} and \cite{gm}   for a wide account and further references on this problem. We point out that such a problem is essentially different from the one studied in this paper: indeed in such a case, the main difficulty is to deal with the set where the solution $u$ is ``small'' in order to prove that the lower order term belongs to $\elle1$. On the contrary it is clear that in our paper we have to deal with the zone where $u$ is ``large'', i.e., where $u$ is near $\sigma$. Recently, a similar problem has been studied in \cite{gs} and \cite{hbv}.

\medskip

Our aim is to prove both existence and nonexistence results for problem \rife{main}. We define a solution as follows.

\begin{dhef}\label{def}\rm
A function $u\in\huz$ is a solution for problem \rife{main} if $0 \leq u<\sigma$ a.e. in $\Omega$, $h(u)|\D u|^2\in \elle{1}$, and
$$
\io M(x)\nabla u \cdot \nabla \vp +\io h(u){|\nabla u |^2} \vp = \io g \vp,
$$
for every $\vp\in \huz\cap\elle{\infty}$.
\end{dhef}

Our first existence result is the following.

\begin{theo}\label{t1}\sl
Assume that $g\in \elle1$, $g \geq 0$, and suppose that 
\begin{equation}\label{rad}
\lim_{s \to \si^{-}}\,\int_{0}^{s} \sqrt{h(t)}\, dt = +\infty\,.
\end{equation}
Then \rife{main} has at least a solution.
\end{theo}

\begin{remark}\label{whye}\rm
Observe that Theorem \ref{t1} covers the case $\gamma\geq 2$ in the model problem \rife{mod}. In general, the fact that if $\sqrt{h}$ does not belong to $L^{1}((0,\si))$, then there exists a solution is not surprising. Indeed, if we define
$$
\Phi(u) = \int_{0}^{u}\,\sqrt{h(s)}\,ds\,,
$$
the condition $h(u)|\D u|^{2}$ in $\elle1$ can be rewritten as
$$
|\D \Phi(u)|^{2} \in \elle1\,,
\quad
\mbox{i.e.,}
\quad
\Phi(u) \in \huz\,.
$$
Thus, by Poincar\'e inequality, $\Phi(u)$ belongs to $\elle2$, and so is almost everywhere finite. Since $\sqrt{h}$ is not in $L^{1}((0,\si))$, this fact implies that $u < \si$ almost everywhere in $\Omega$. In other words, the condition ``$u$ does not touch $\si$'' is automatically satisfied starting from the formulation. Note that if $\sqrt{h}$ belongs to $L^{1}((0,\si))$, then $u$ may be well equal to $\si$ on a set of positive measure without any contradiction, since in this case the function $\Phi(u)$ is finite no matter the values taken by $u$. As we will see, this fact will lead to nonexistence of solutions for \rife{main} if the datum $g$ is ``large'' (recall that we need for a solution to be almost everywhere smaller than $\si$).
\end{remark}

Thanks to the result of Theorem \ref{t1}, we are left to deal with the case of $\sqrt{h}$ in $L^{1}((0,\si))$.
This corresponds to the case $0<\gamma<2$ in the model case \rife{mod}. 
To better understand what happens in this case, let us give a one-dimensional example.

\begin{example}\label{gauno}\rm
Consider the following problem:
\be\label{pb1d}
\bc
-u''(s) + \frac{(u'(s))^{2}}{1 - u(s)} = \la & \mbox{in $(-1,1)$,} \\
\hfill u(-1) = u(1) = 0\,, \hfill
\ec
\ee
and suppose that \rife{pb1d} has a $C^{2}$ solution for every $\la > 0$, with $u(s) < 1$ almost everywhere in $(-1,1)$. If $\la = 6$ there exists an explicit solution: $w(s) = 1 - s^{2}$. Let now $u$ be a solution of \rife{pb1d} with $\la > 6$; multiplying the equation by $1-u$ we obtain
$$
-(1-u)u'' + (u')^{2} = \la\,(1-u)\,,
$$
which can be rewritten as
$$
-\left( u - \frac{u^{2}}{2}\right)'' = \la\,(1-u)\,.
$$
Defining $v = u - u^{2}/2$ (and, correspondingly, $z = w - w^{2}/2$), we have
$$
\bc
-v''(s) = \la\sqrt{1 - 2v(s)} \\
\hfill v(-1) = v(1) = 0\,, \hfill
\ec
\quad
\mbox{and}
\quad
\bc
-z''(s) = 6\sqrt{1 - 2z(s)} \\
\hfill z(-1) = z(1) = 0\,. \hfill
\ec
$$
Subtracting the first equation from the second, multiplying by $(z-v)^{+}$ and integrating on $(-1,1)$ yields (since $\la > 6$)
$$
\int_{-1}^{1}\,|((z-v)^{+})'|^{2}\,ds
\leq
\la\int_{-1}^{1}\,[\sqrt{1-2z} - \sqrt{1-2v}](z-v)^{+}\,ds \leq 0\,.
$$
Therefore, $((z-v)^{+})' \equiv 0$, which implies $(z-v)^{+} \equiv 0$. Hence, $v(s) \geq z(s)$,
and so $u(s) \geq w(s) = 1 - s^{2}$ in $[-1,1]$, which implies $u(0) = 1$. Since $u$ belongs
to $C^{2}$, then $u(s) = 1 + \frac{u''(0)}{2}s^{2} + {\rm o}(s^{2})$, and we can use the equation to deduce, thanks to the de l'H\^{o}pital rule, 
$$
\la = \lim_{s \to 0^{+}}\,\left[-u''(s) + \frac{(u'(s))^{2}}{1 - u(s)} \right]
= \lim_{s \to 0^{+}}\,-3u''(s) = -3u''(0)\,.
$$
Therefore, $u(s) = 1 - \frac{\la}{6}s^{2} + {\rm o}(s^{2})$. Since $\la > 6$, 
this fact contradicts $u(s) \geq 1 - s^{2}$, and so \rife{pb1d} can not have solutions such 
that $u < 1$ almost everywhere for every $\la > 6$. On the other hand, if $\la < 6$ then any 
solution $u$ of \rife{pb1d} is smaller than $w$ (with the same proof as above), and so satisfies $u(s) \leq 1 - s^{2}$. If $u(0) = 1$ then we have $u''(0) = -\frac{\la}{3}$ and so $u(s) = 1 - \frac{\la}{6}s^{2} + {\rm o}(s^{2})$, a fact that contradicts $u(s) \leq 1 - s^{2}$ since $\la < 6$. Therefore, any solution $u$ of \rife{pb1d} with $\la < 6$ satisfies $u(s) \leq u(0) < 1$. In other words, we have that any solution $u$ of \rife{pb1d} is strictly smaller than~1 if $\la < 6$, ``touches''~1 at the origin if $\la > 6$, and cannot be almost everywhere smaller than~1 if $\la > 6$.

Note that if we consider \rife{pb1d} in $(-R,R)$, we have by rescaling that
$w(s) = 1 - (s/R)^{2}$ is a solution with $6/R^{2}$ as datum, and that we have a contradiction
with existence of solutions if $\la > 6/R^{2}$.
\end{example}

In view of the preceding example, if \rife{rad} is not satisfied one can expect 
that \rife{main} has always solutions if the size of $g$
is small, 
while no solutions are expected if the size of $g$ is large enough (depending both on $h$ and on the size of $\Omega$). 

Let us clarify this fact by introducing the following 

\begin{dhef}\label{capla}\rm
Let $f$ in $\elle p$, $p > \frac N2$, $f \geq 0$, $\la > 0$, and consider the following problem
\be\label{mainla}
\bc
-\dive(M(x)\D u) + h(u)|\D u|^2=\la\, f & \mbox{in $\Omega$},\\
\hfill u=0 \hfill & \mbox{on $\partial\Omega$.}
\ec
\ee
We define $\Lambda_{f}$ as follows:
$$
\Lambda_{f}=\sup \{\la >0: \exists u\leq\si-\eps \mbox{, for some $\eps > 0$, $u$ solution of } \rife{mainla}\}\,, 
$$
where we use the convention that $\sup \emptyset = 0$.  
\end{dhef}

\begin{remark}\label{supnonmax}\rm
The definition of $\Lambda_{f}$ may seem strange, since we require the solutions to \rife{mainla} to be strictly smaller than $\sigma$, a fact that is not true for the solution $u = 1 - s^{2}$ obtained in Example \ref{gauno} for $\la = 6$. This restriction is purely technical, is needed for the proof of Theorem \ref{t2+} below, and, as in the case of the example above, may yield that $\Lambda_{f}$ is an actual supremum, rather than a maximum. Note indeed that in Example \ref{gauno} we have
$$
\{\la >0: \exists u\leq\si-\eps \mbox{, for some $\eps > 0$, $u$ solution of } \rife{pb1d}\} = (0,6)\,.
$$
See also Section~6 for further remarks on $\Lambda_{f}$.
\end{remark}

We have the following existence result for $\la$ small (without growth assumptions on $h$).

\begin{theo}\label{t2}\sl
Let $f \in \elle{p}$, $p>\frac{N}{2}$, $f \geq 0$. Then there exists $\Lambda$ such that for every $\la < \Lambda$ problem \rife{mainla} has a solution (smaller than $\si-\eps$ for some $\eps = \eps(\la) > 0$). Therefore, $\Lambda_{f} > 0$.
\end{theo}

In the particular case in which the principal part is the laplacian, we can be much more precise. 

\begin{theo}\label{t2+}\sl
Let $f$ be as in Theorem \ref{t2} and let $M(x) \equiv I$, the identity matrix. Then \rife{mainla} has a unique solution $u_{\la}$ for every
$\lambda<\Lambda_f$. Moreover, $u_{\la} \leq u_{\mu}$ if $\la < \mu$.
\end{theo}

In order to prove the preceding result, we will compare solutions of \rife{mainla}, with $M(x) \equiv I$, that correspond to different values of $\lambda>0$. In the general case, comparison principles for $\huz$ solutions of \rife{mainla} are difficult to obtain, and to our knowledge the only results in this direction are contained in \cite{as} and \cite{bamu} (see also Remark \ref{compar} below).

\medskip

If, in addition, $\sqrt{h}$ belongs to $L^{1}((0,\si))$) we will show that $\La_{f} < +\infty$. 
In order to give a sharper statement of this result, let us define $\la_{1}(f)$ and $\vp_{1}(f)$ as the first eigenvalue and the first eigenfunction of $-\dive(M(x)\D u)$ in $\Omega$ with weight $f\gvertneqq 0$ in $\elle p$, $p > \frac N2$, i.e.,
$$
\bc
-\dive(M(x)\D\vp_{1}(f)) = \la_{1}(f)f(x)\vp_{1}(f) & \mbox{in $\Omega$,}\\
\hfill \vp_{1}(f) = 0 \hfill & \mbox{on $\partial\Omega$.} 
\ec
$$
We recall that $\vp_{1}(f)$ belongs to $\elle\infty$, that $\vp_{1}(f) > 0$ in $\Omega$, and that
\be\label{poincar}
\io \,M(x)\D u \cdot \D u \geq \la_{1}(f)\,\io \,f(x)\,u^{2}\,,
\quad
\forall u \in \huz\,.
\ee
We also define ($\alpha$ is given by \rife{a1}):
\be\label{hpsi}
H(s) = \frac 1\alpha\,\int_{0}^{s}\,h(t)\,dt\,,
\quad
\mbox{and}
\quad
\psi(u) = \int_{0}^{s}\,{\rm e}^{-H(t)}\,dt\,.
\ee

\noindent Our first nonexistence result is the following.

\begin{theo}\label{nnhl1}\sl
Let $f$ be in $\elle{p}$, $p > \frac N2$, $f \gvertneqq  0$.
Suppose that $h$ belongs to $L^{1}((0,\si))$, i.e., that
$$
H(\si) = \frac{1}{\al}\,\int_{0}^{\si}\,h(s)\,ds<+\infty\,,
$$
and let $\la > \la_{1}(f) {\rm e}^{H(\si)}\psi(\si)$.
Then there exists no weak solution for \rife{mainla}. Therefore, $\La_{f} < +\infty$.
\end{theo}

In view of Theorem \ref{nnhl1}, we are left with the case
$\sqrt{h}$ in $L^{1}((0,\si))$, but $h$ not in $L^{1}((0,\si))$ (i.e. $1\leq \gamma<2$ in \rife{mod}). 
In this case we have to make additional assumptions on $h$, $M$ and $f$.

\begin{theo}\label{nhnl1}\sl
Let $f$ be in $\elle{p}$, $p > \frac N2$.
Let $M(x) \equiv I$, and suppose that $h$ and $f$ are such that
\be\label{assumpt2}
\bc
\mbox{there exists } \ga \in [1,2) : \lim\limits_{s \to \si^{-}}\,(\si-s)^{\ga}\,h(s) = C > 0\,, \\
\hfill \mbox{there exists } \rho > 0 : f(x) \geq \rho \mbox{ in $\Omega$\,.} \hfill
\ec
\ee
Then \rife{mainla} has no solutions for $\la$ large enough, so that $\La_{f} < +\infty$.
\end{theo}

As a consequence of theorems \ref{t1}, \ref{t2}, \ref{t2+}, \ref{nnhl1} and \ref{nhnl1}, we have a complete picture for the model example \rife{mod}:
\begin{itemize}
\item[i)] if $\ga \geq 2$, \rife{mod} has a solution in $\huz \cap \elle\infty$ for every $g$ in $\elle1$, $g \geq 0$;

\item[ii)] if $g = \la\,f$, with $f$ in $\elle{p}$, $p > \frac{N}{2}$, \rife{mod} has a solution for every $\la < \La_{f}$;

\item[iii)] if $g = \la\,f$, with $f$ in $\elle{p}$, $p > \frac{N}{2}$, and $0 < \ga < 2$, then $\La_{f} < +\infty$, and \rife{mod} has no solutions for $\la > \La_{f}$.
\end{itemize}

\medskip

The plan of the paper is as follows. In Section~2 we will prove Theorem \ref{t1} by 
approximating \rife{main} with a sequence of nonsingular problems, and in Section~3 we will prove 
Theorem 
\ref{t2} and Theorem \ref{t2+}. In Section~4 we will deal with nonexistence results, proving theorems 
\ref{nnhl1} and \ref{nhnl1}. In order to prove this latter result, we will transform \rife{mainla} into a 
semilinear problem which has solutions for every $\la > 0$, and prove (using one-dimensional 
analysis and super- and sub-solution techniques) that if $\la$ is large enough these 
solutions have ``flat'' zones of nonzero measure which correspond to zones where $u \equiv \si$. 
The final Section~5 will be devoted to the study of the asymptotic behaviour of sequences of
solutions of approximating problems if the limit problem \rife{main} has no solution, while in Section~6 we will study the case $\la = \Lambda_{f}$.

\medskip

{\bf Notation.} We will use the following notation throughout the paper: if $k > 0$ we define 
$$
T_k (s)= \max ( -k, \min ( s, k))\,,
\quad
G_k (s)=s-T_k (s)\,,
$$
and by $\eps_n$ we indicate any quantity that tends to $0$ as $n$ tends to infinity.

\section{Proof of Theorem \ref{t1}}

Our approach to prove the existence of a solution of \rife{main} is by approximation. We will consider the sequence $\{u_{n}\}$ of solutions of 
\be\label{appn}
\bc
-\dive ( M(x)\D u_n) + {h_n}(u_n)|\D u_n|^2= g_n & \mbox{in $\Omega$,}\\
\hfill u_n=0 \hfill & \mbox{on $\partial\Omega$,}
\ec
\ee
where $g_n =T_n (g)$, and $h_n$ is defined as
\be \label{hn}
h_n (s)= \begin{cases}
\hfill 0 \hfill & \mbox{if $s < 0$,} \\
\hfill n h(\frac1n)s \hfill & \mbox{if $0 \leq s < \frac1n$,} \\
\hfill h(s) \hfill & \mbox{if $\frac1n \leq s < \si$ and $h(s) \leq n$,} \\
\hfill n \hfill & \mbox{if $0 \leq s < \si$ and $h(s) > n$ or if $s \geq \si$.}
\end{cases}
\ee

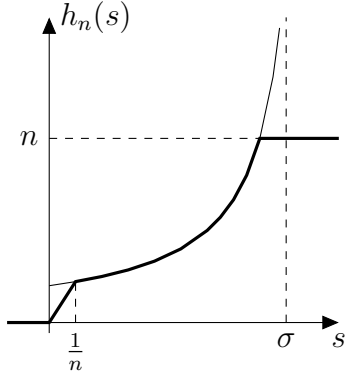
\begin{center}
\begin{tikzpicture}[scale=0.7,>=triangle 45]
\draw[->] (-0.5,0) -- (5.5,0) node[below]{$s$};
\draw[->] (0,-0.2) -- (0,5.8) node[right]{$h_{n}(s)$};
\draw[dashed] (4.5,0) node[below]{$\sigma$} -- (4.5,5.8);
\draw (0,0.7) -- (0.5,0.7/0.9) -- (1,0.7/0.8) -- (1.5,0.7/0.7) -- (2,0.7/0.6) -- (2.5,0.7/0.5) -- (3,0.7/0.4) -- (3.25,0.7/0.35) -- (3.5,0.7/0.3) -- (3.75,0.7/0.25) -- (4,0.7/0.2) -- (4.25,0.7/0.15) -- (4.375,0.7/0.125);
\draw[very thick] (-0.8,0) -- (0,0) -- (0.5,0.7/0.9) -- (1,0.7/0.8) -- (1.5,0.7/0.7) -- (2,0.7/0.6) -- (2.5,0.7/0.5) -- (3,0.7/0.4) -- (3.25,0.7/0.35) -- (3.5,0.7/0.3) -- (3.75,0.7/0.25) -- (4,0.7/0.2) -- (5.5,0.7/0.2);
\draw[dashed] (0,3.5) node[left]{$n$} -- (4,3.5); 
\draw[dashed] (0.5,0) node[below]{$\frac1n$} -- (0.5,0.7/0.9); 
\end{tikzpicture}
\end{center}

\begin{remark}\label{h00}\rm
The definition of $h_{n}(s)$ in $[0,\frac1n)$ is needed to have that $h_{n}(s)\,s \geq 0$ for every $s$ in $\re$, and that $h_{n}(s)$ is continuous at~0. If $h(0) = 0$, there is no need to define $h_{n}(s)$ as above on $[0,\frac1n)$: it is enough to take $h_{n}(s) = h(s)$ on this set.
\end{remark}

Since $h_{n}(s)\,s \geq 0$, and $h_{n}$ is bounded, by a result of \cite{bbm} there exists a solution $u_{n}$ of \rife{appn}, i.e., a function $u_{n}$ in $\huz$ such that $h_n(u_n){|\nabla u_n |^2}$ belongs to $\elle1$, and such that 
\be\label{weakn}
\io M(x)\D u_n \cdot \D \vp +\io h_n(u_n){|\nabla u_n |^2} \vp = \io g_n\, \vp\,,
\ee
for every $\vp\in \huz\cap\elle{\infty}$. Moreover, since $g_{n}$ belongs to $\elle\infty$, and since $h_{n}(s) \, s \geq 0$, we have that $u_{n}$ belongs to $\elle\infty$.

In order to prove Theorem \ref{t1}, we begin by proving some properties of the sequence $\{u_{n}\}$.

\begin{prop} \label{prop}\sl
Consider the sequence $\{u_{n}\}$ of solutions of \rife{appn} with $g\in \elle1$, $g \geq 0$. Then:
\begin{enumerate}
\item[i)] $\{u_n\}$ is nonnegative and bounded in $\huz$; consequently, it weakly converges (up to subsequences) to some function $u$ in $\huz$;
\item[ii)] 
\be\label{stimint}
\io h_n (u_n) |\D u_n |^2 \leq \| g\|_{\elle1}\,;
\ee
\item[iii)] $0 \leq u \leq \si$, almost everywhere in $\Omega$;
\item[iv)] for every $k$ in $(0,\si)$, $T_k (u_n)$ strongly converges to $T_k (u)$ in $\huz$;
\item[v)] $u_n$ strongly converges to $u$ in $\huz$;
\item[vi)] If $\mis{(\{u = \si\})} = 0$, then $h(u_{n})|\D u_{n}|^{2}$ converges to $h(u)|\D u|^{2}$ almost everywhere in $\Omega$.
\end{enumerate}
\end{prop}

\begin{proof}[{\sl Proof}]  i) First of all, observe that since $g_{n} \geq 0$, the fact that $h_{n}(s)\,s \geq 0$ implies, by standard arguments, that $u_{n} \geq 0$ a.e.\ in $\Omega$.

Choosing $\vp =T_\si (u_n)$ as test function in \rife{weakn} (as in \cite{bg3}) we obtain
$$
\io M(x) \D u_n\cdot \D T_\si (u_n) + \io h_n(u_n)|\D u_n|^2\, T_\si (u_n)= \io
T_\si (u_n) g_n \,.
$$
Using \rife{a1}, the fact that $g_n\leq g$, and since both $T_{\si}(u_{n}) = \si$ 
and $h_n (u_n)=n$ on the set $\{u_{n} \geq \si\}$, we have
$$
\al \int_{\{u_n\leq \si\}}| \D u_n |^2 + n\si \int_{\{u_n\geq \si\}} |\D u_n|^2 \leq \si \|g\|_{\elle1}\,,
$$
which implies that
$\{u_n\}$ is bounded in $H^1_0 (\Omega)$.
Therefore, there exists $u$ in $H^1_0 (\Omega)$ such that (up to subsequences) $u_n$ converges to $u$ weakly in $H^1_0(\Omega)$.

\medskip

ii) Let $\eps>0$, let $0\leq k<\si$, and choose $\vp = {\frac{1}{\eps}}T_\eps (G_k (u_n))$ as test function in \rife{weakn}. We obtain
$$
\begin{array}{l}
\dys
{\frac{1}{\eps}} \io M(x) \D u_n\cdot \D T_\eps (G_k (u_n) )
\\
\dys
\quad
+ {\frac{1}{\eps}}\io h_n(u_n)|\D u_n|^2 T_\eps (G_k (u_n))
= \io {\frac{1}{\eps}}T_\eps (G_k (u_n)) g_n\,.
\end{array}
$$
Dropping the first (nonnegative) term we have, since $T_{\eps}(G_{k}(u_{n})) = \eps$ where $u_{n} \geq k+\eps$, and since $0 \leq g_{n} \leq g$,
$$
\dys \int_{\{ u_n \geq k+\eps\} } h_n(u_n)|\D u_n|^2 \leq \int_{\{u_n \geq k\}} g_{n} \leq \int_{\{u_n \geq k\}} g \,.
$$
Taking the limit as $\eps$ tends to $0^{+}$, we have 
\be\label{diseg}
\dys \int_{\{ u_n \geq k\} } h_n(u_n)|\D u_n|^2 \leq \int_{\{u_n \geq k\}} g\,,
\ee
which then gives \rife{stimint} taking $k = 0$. Since
$$
-\dive(M(x)\D u_{n}) = g_{n} - h_{n}(u_{n})|\D u_{n}|^{2}\,,
$$
and the right hand side is bounded in $\elle1$ as a consequence of \rife{stimint} and of the assumptions on $g$, we obtain from a result of \cite{BM} that (up to subsequences) $\D u_{n}$ converges to $\D u$ almost everywhere in $\Omega$.

\medskip

iii) From \rife{diseg}, and the fact that $h_{n}$ is increasing, we deduce
\be\label{1}
\dys \int_{\{u_n \geq k\} }|\D u_n|^2 \leq \frac{1}{h_n(k)}\,\io g\,.
\ee
Choosing $k=\si$ we have
\be\label{verysmall}
\dys \int_{\{ u_n \geq \si\} }|\D u_n|^2 \leq \frac{1}{ n}\, \io g\,. 
\ee
Letting $n$ tend to infinity, and using Fatou lemma together with the almost everywhere convergence of $\D u_{n}$, we have
$$
\int_{\{ u > \si\} }|\D u|^2 = 0,
$$
which implies $0 \leq u \leq \si$ almost everywhere in $\Omega$.

\medskip

iv) Let $0 < k < \si$, and
$$
\vp_\eta (s) = se^{\eta s^2}\,,\quad \eta >0\,.
$$
The function $\vp_{\eta}$ has the following property:
\be\label{vpla}
\alpha \vp_{\eta}'(s) - \frac{\beta h (k)}{\alpha} |\vp_{\eta}(s)| \geq \frac{\alpha}{2}\,,
\quad
\forall \eta >\frac{\beta^2 h^2(k)}{4\alpha^4}\,,
\ 
\forall s \in \re
\,.
\ee
Hence, we fix $\eta >\frac{\beta^2 h^2(k)}{4\alpha^4}$, and choose 
$\vp = \vp_{ \eta } (T_k (u_n)-T_k (u))$ as test function in \rife{weakn}. 
We obtain (for the sake of brevity, we omit the arguments from $\vp_{\eta}$ and $\vp'_{\eta}$)
\be\label{truncates}
\begin{array}{l}
\dys
\io M(x)\D u _n \cdot \D (T_k (u_n)-T_k (u)) \vp_{\eta}'\\
\dys
\quad
+ \io h_n(u_n)|\D u_n|^2 \vp_{\eta} 
= \io
\vp_{\eta} g_{n}\,.
\end{array}
\ee
Since $u_n$ converges to $u $ a.e., and since $|T_k (u_n )-T_k (u)|\leq 2k $, we have
$$
\io
\vp_{\eta} (T_k (u_n)-T_k (u))g_{n} = \eps_n\,.
$$
Since $T_k (u_n) \to T_k (u)$ a.e. and weakly in $\huz$ and since $\vp_{\eta}'$ is bounded, it follows that 
$$
-\io M(x)\D T_k (u) \cdot \D (T_k (u_n)-T_k (u)) \vp_{\eta}' =\eps_n\,.
$$
Thus  adding such a quantity on both sides of \rife{truncates}, 
 dropping the nonnegative term 
$$
\int_{\{u_{n} > k\}}h_{n}(u_n)|\D u_n|^2 \vp_{\eta}\,,
$$
and using \rife{a1}, we have (recall that $u_{n} = T_{k}(u_{n}) + G_{k}(u_{n})$)
$$
\begin{array}{l}
\dys
\io M(x) \D (T_k (u_n) - T_k ( u)) \cdot \D (T_k (u_n)-T_k (u)) \vp_{\eta}'
\\
\dys
\qquad
+ \int_{\{u_n\leq k\}} h_{n}(u_n)|\D u_n|^2 \vp_{\eta} 
\\
\dys
\quad
\leq
\eps_n
+
\beta \io |\D G_k (u_n)| | \D T_k (u)| |\vp_{\eta}' (k-T_k (u))|\,.
\end{array}
$$
We note that, since $h_{n} \leq h$, since $h_{n}$ is increasing, and by \rife{a1},
$$
\left| \int_{\{u_n\leq k\}} h_{n}(u_n)|\D u_n|^2 \vp_{\eta} \right|
\leq 
\frac{h(k)}{\alpha} \io M(x) \D T_k (u_n) \cdot \D T_k (u_n)| \vp_{\eta}|\,.
$$ 
Since we have (by \rife{a1})
$$
\begin{array}{l}
\dys
\io M(x) \D T_k (u_n) \cdot \D T_k (u_n)| \vp_{\eta} |
\\
\dys
\quad
=
\io M(x) \D (T_k (u_n)-T_k (u)) \cdot \D (T_k (u_n) - T_k (u))| \vp_{\eta} |
+\eps_n
\\
\dys
\quad
\leq
\beta \io |\D (T_k (u_n)-T_k (u))|^{2}|\vp_{\eta}| + \eps_n
\,,
\end{array}
$$
and
$$
\io |\D G_k (u_n)| | \D T_k (u)| |\vp_{\eta}' (k-T_k (u))| =\eps_n\,,
$$
we conclude that, using \rife{a1},
$$
\io |\D (T_k (u_n) - T_k ( u))|^2 \left[ \alpha \vp_{\eta}' - \frac{\beta h (k)}{\alpha} |\vp_{\eta}|\right]
\leq 
\eps_n \,.
$$
Using \rife{vpla}, we therefore deduce that 
$$
\frac{\alpha}{2} \io |\D (T_k (u_n) - T_k ( u))|^2 
\leq 
\eps_n\,,
$$
and thus $T_k (u_n)$ strongly converges to $T_k (u)$ in $H^1_0 (\Omega)$.

\medskip

v) Since $\D u_n$ converges to $\D u$ almost everywhere in $\Omega$, it is enough to prove that $\{|\D u_n |^2\}$ is equiintegrable. Let
$E\subset \Omega$ be measurable, and let $\eps>0$. Using \rife{1}, we have, for $0 < k < \si$,
$$
\int_{E} |\D G_{k}(u_{n})|^{2}
\leq
\io |\D G_{k}(u_{n})|^{2}
=
\int_{\{u_{n} \geq k\}} |\D u_{n}|^{2}
\leq
\frac{1}{h_{n}(k)}\io g\,.
$$
By the assumptions on $h$, there exist $k_{\eps} > 0$ and $n_{\eps} > 0$ such that for every $n\geq n_{\eps}$ we have
$$
\int_{E} |\D G_{k_{\eps}}(u_{n})|^{2} \leq \frac{\eps}{2}\,. 
$$
Once $k_{\eps}$ is fixed, by iv) and by Vitali convergence theorem we have 
$$
\dys \int_E |\D T_{k_{\eps}} (u_n) |^2 \leq \frac{\eps}{2}\,, \ \ \forall\ n\in \na\,,
$$
if the measure of $E$ is small enough. This concludes the proof of v), since $u_{n}=T_{k_{\eps}} (u_{n})+G_{k_{\eps}} (u_{n})$.

\medskip

vi) Even if we already know that both $u_{n}$ and $\D u_{n}$ converge almost everywhere 
in $\Omega$, the fact that $h(0)$ can be strictly positive means that we have to be careful when dealing with the set $\{u=0\}$. Anyway, as we are going to show, the presence of the quadratic gradient term will allow us to conclude using a result by G. Stampacchia (see \cite{st2}).

\medskip

Define $\Omega' = \{x \in \Omega : u(x) < \si\}$; by assumption, and by iii), we have $\mis(\Omega) = \mis(\Omega')$. Let
$$
E_u = \{x\in\Omega': u_n(x)\not\rightarrow u(x)\}\,,
\ \ 
E_{\nabla u} = \{x\in\Omega': \nabla u_n(x)\not\rightarrow \nabla u(x)\}\,,
$$
so that $\mis(E_u \cup E_{\nabla u})=0$. Thus, if we define $\Omega'' = \Omega'\backslash (E_u \cup E_{\nabla u})$, we have that $\mis(\Omega'') = \mis(\Omega)$. 

Define now
$$
F_+ = \{x\in\Omega'': u(x)>0\}\,,
\quad
F_0= \{x\in\Omega'': u(x)=0, \nabla u(x)=0\}\,.
$$
Since, by a result by G. Stampacchia, we have that $\D u = 0$ almost everywhere on the set $\{u = 0\}$, we have
$$
\mis(F_+ \cup F_0) = \mis(\Omega'') = \mis(\Omega)\,.
$$
Now, if $x\in F_+$, then $h_n (u_n(x))|\D u_n(x)|^2$ tends to 
$h (u (x))|\D u(x)|^2$, while if $x\in F_0$, then, since $h_n (u_n(x))$ is bounded and $\nabla u_n (x)$ converges to $\nabla u(x) = 0$, we have 
$$
\lim_{n \to +\infty}\,h_n(u_n(x))|\nabla u_n (x)|^2 = 0 = h(u(x))|\nabla u (x)|^2\,. 
$$ 
Therefore, $h_n(u_n(x))|\nabla u_n (x)|^2$ converges to $h(u(x))|\nabla u (x)|^2$ in $F_{+} \cup F_{0}$, i.e., almost everywhere in $\Omega$.
\end{proof} 

Observe that the results of Proposition \ref{prop} are not enough to prove the existence of a solution for \rife{main}. Indeed, to pass to the limit in \rife{weakn} we need the strong compactness of the lower order term in $\elle1$. 
In order to prove this fact, we need an information about the measure of the set in which $u$ is close to $\si$.

\begin{lemma}\label{cpt}\sl
Let $\{u_{n}\}$ be a sequence of solutions of \rife{appn}, with $g \in \elle1$, $g \geq 0$, and suppose that 
\be\label{meas}
\bc
\hfill \forall \de>0\ \exists \tau > 0\,,\ \exists n_0>0:\hfill \\
\mis (\{ \si-\tau\leq u_n \leq \si+\tau\}) \leq \de\,,\quad \forall n\geq n_0 \,.
\ec
\ee
Then $h_n (u_n) |\D u_n|^2$ is strongly compact in $\elle1$.
\end{lemma}

\begin{proof}[{\sl Proof}]  We already know, by Proposition \ref{prop}, that, up to subsequences, $u_{n}$ almost everywhere converges in $\Omega$ to some function $u$, with $u \leq \si$ almost everywhere. We begin to prove that, under assumption \rife{meas}, $\mis{(\{u = \si\})} = 0$. Indeed, we have
$$
\liminf_{n \to +\infty}\,\chi_{\{\si-\tau \leq u_{n} \leq \si+\tau\}}(x)
\geq
\chi_{\{\si-\tau < u < \si+\tau\}}(x)\,,
$$
almost everywhere in $\Omega$. This, Fatou Lemma and \rife{meas} imply
$$
0 \leq \mis (\{u = \si\}) \leq \mis{(\{\si-\tau < u < \si+\tau\})} \leq \de\,,
$$
and thus $\mis{(\{u = \si\})} = 0$. This fact, and vi) of Proposition \ref{prop} imply that $h_n (u_n) |\D u_n|^2$ almost everywhere convergest to $h(u) |\D u|^2$ in $\Omega$. Thus, to apply Vitali theorem, we only need to prove that the sequence $\{h_n (u_n)|\D u_n|^2 \}$ is equiintegrable. For every $E\subset \Omega$ measurable and for every $0<k<\si$ we have, since $h_{n}$ is increasing, and $h_{n}(s) \leq h(s)$ if $s > \frac 1n$, 
$$
\begin{array}{r@{\hspace{2pt}}c@{\hspace{2pt}}l}
\dys 
\int_E 
h_n (u_n)|\D u_n|^2 
& = &
\dys
\int_{E\cap\{u_n \leq k\}} h_n (u_n)|\D u_n|^2 +
\int_{E\cap\{u_n \geq k\}} h_n (u_n)|\D u_n|^2 
\\
& \leq &
\dys 
h(k)\int_{E} |\D T_k (u_n) |^2 
+ 
\int_{\{u_n \geq k\}} h_n (u_n)|\D u_n|^2 \,.
\end{array}
$$
Using \rife{diseg}, we have that 
$$
\int_{\{u_n \geq k\}} h_n (u_n)|\D u_n|^2 
\leq
\int_{\{k\leq u_n \leq \sigma \}} g
+ \int_{\{ u_n > \sigma \}} g\,,
$$
and the last integral tends to zero as $n$ tends to infinity since $u \leq \si$ almost everywhere in $\Omega$ (by iii) of Proposition \ref{prop}). 
Let now $\eps>0$; since $g$ belongs to $\elle1$, there exists $\de_{\eps} > 0$ such that
$$
\mis(E) < \de_{\eps} \ \Rightarrow\ \int_{E}\,g < \frac{\eps}{2}\,.
$$
Let $k_{\eps}<\si$ be such that \rife{meas} holds true with $\de_{\eps}$; therefore,
$$
\int_{\{k_{\eps} \leq u_{n} \leq \si\}}\,g < \frac{\eps}{2}\,,
\quad
\forall n \geq n_{0}\,.
$$
Once $k_{\eps} < \si$ is fixed, we have that $T_{k_{\eps}} (u_n)$ is strongly compact in $\huz$ by iv) of Proposition \ref{prop}; therefore, we can choose $\mis(E)$ small enough so that 
$$
h (k_{\eps})\int_{E} |\D T_{k_{\eps}} (u_n) |^2 \leq \frac{\eps}{2}\,,
$$
uniformly with respect to $n$. By applying Vitali theorem the conclusion then follows. \end{proof} 

We can now prove Theorem \ref{t1}. In view of Lemma \ref{cpt}, we are going to prove that if $\sqrt{h}$ does not belong to $L^{1}((0,\si))$, then \rife{meas} holds true.

\begin{proof}[{\sl Proof of Theorem \ref{t1}}] Thanks to Proposition \ref{prop} we have
\be\label{inizio}
\io\,h_n (u_n)|\D u_n |^2
\leq
\|g\|_{\elle1}\,.
\ee
Defining $\Phi_n(s)=\int_0^s \sqrt{h_n(t) } dt $,
we can write the above inequality as
$$
\io |\D \Phi_n (u_n) |^2
\leq
\|g\|_{\elle1}\,.
$$
From Poincar\'e inequality we then deduce, for every $\tau > 0$, and for some $C > 0$,
$$
C \int_{\{\si-\tau \leq u_n\leq \si+\tau \} } | \Phi_{n} (u_n) |^2
\leq
C \io | \Phi_{n} (u_n) |^2
\leq \|g\|_{\elle1}\,.
$$
Thus, since $\Phi_{n}$ is increasing, we have 
$$
\mis (\{\si-\tau \leq u_n\leq \si+\tau \}) \leq
\frac{\|g\|_{\elle1}}{C| \Phi_n (\si-\tau) |^2}
\,.
$$
If $n$ is large enough, then $h_{n}(s) = h(s)$ on $[0,\si-\tau]$, so that
$$
\mis (\{\si-\tau \leq u_n\leq \si+\tau \}) \leq
\frac{\|g\|_{\elle1}}{C| \Phi (\si-\tau) |^2}
\,,
$$
where $\Phi(t) = \int_{0}^{t}\,\sqrt{h(s)}\,ds$. Since $\Phi$ is unbounded on $[0,\si)$ by \rife{rad}, \rife{meas} follows from the above inequality.
\end{proof}

\begin{remark}\label{abs} \rm 
For the sake of simplicity, we chose to present the existence 
result of Theorem \ref{t1} for problem \rife{main}, even if several generalizations could be possible. For instance, in the proof we have never used the linearity of the principal part with respect to the gradient, so that it is easy to see that there exists a solution for 
$$
\bc
A(u) + H(x,u,\D u) = g & \mbox{in}\,\Omega\,\\
\hfill u=0 \hfill & \mbox{on $\partial\Omega$,}
\ec
$$
where $A(u) = -\dive (a(x,\D u))$ is a pseudomonotone operator (see \cite{ll} for more details), and $H$ is such that $H(x,s,0) \equiv 0$, and 
$$
h_1(s)|\xi|^2 \leq H(x,s,\xi) \leq h_2 (s)|\xi|^2\,,
$$
with $h_1$ and $h_{2}$ continuous, increasing functions such that \rife{h1}, \rife{h2}, and \rife{rad} hold true.
\end{remark}

\section{Proof of theorems \ref{t2} and \ref{t2+}}

Let $\{u_{n}\}$ be the sequence of solutions of 
\be\label{app}
\bc
-\dive ( M(x)\D u_n) + {h_n}(u_n)|\D u_n|^2= \la f\quad & \mbox{in $\Omega$,}\\
\hfill u_n=0\hfill & \mbox{on $\partial\Omega$,}
\ec
\ee
where
$
h_n (s)$ has been defined in \rife{hn}. Such solutions exist by a result in \cite{bbm} and are, thanks to the assumptions on $f$,  in $\huz \cap \elle\infty$ for every fixed $n$. In order to prove Theorem \ref{t2}, we will use the   summability  assumption on $f$ to prove that the sequence $\{u_{n}\}$ of solutions of \rife{app} is uniformly bounded in $\elle\infty$ by a constant (depending on $\la$) which can be made strictly smaller than $\si$ by choosing $\la$ small. So that, roughly speaking, we deal with an equation which is no longer singular.
This will allow us to apply Lemma \ref{cpt} to pass to the limit in \rife{app}.

\begin{proof}[{\sl Proof of Theorem \ref{t2}}]
Let $k > 0$, and choose $G_{k}(u_{n})$ as test function in \rife{app}. Using the fact that $u_{n} \geq 0$, and that $h_{n}(s)$ is nonnegative, we have
$$
\io M(x)\D u_{n} \cdot \D G_{k}(u_{n}) \leq \lambda\io f\,G_{k}(u_{n}).
$$ 
Then we use \rife{a1} to deduce, thanks to a classical result by G. Stampacchia (see \cite{st}), that there exists a constant 
$C_0>0$ (depending only on $\Omega$, $\alpha$, $p$ and $N$) such that 
\be\label{lio}
\norma{u_n}{\elle\infty} \leq \lambda \, C_0\, \norma{f}{\elle p}\,.
\ee
Consequently, if $\la<\Lambda = \frac{\sigma}{\|f\|_{\elle{p}} C_0}$, \rife{lio} implies \rife{meas}, and so both Proposition \ref{prop} and Lemma \ref{cpt} can be applied, yielding the existence of a solution of \rife{mainla} for $\la$ small. Therefore, $\La_{f} > 0$. 
\end{proof}

\begin{proof}[{\sl Proof of Theorem \ref{t2+}}]
We recall that now $M(x) \equiv I$. We are going to prove that \rife{mainla} has a solution for every $\la < \La_{f}$. Let $0 < \la < \La_{f}$, and let $\mu$ in $(\la,\La_{f})$ be such that \rife{mainla} has a solution $v$ such that $0 \leq v \leq \si-\eps$ for some $\eps > 0$. Define
$$
\overline{h}(s) = 
\begin{cases}
\hfill h(s) \hfill & \mbox{if $0 \leq s \leq \si - \eps$,} \\
\hfill h(\si-\eps) \hfill & \mbox{if $s > \si - \eps$.}
\end{cases}
$$
Since $\overline{h}$ is bounded and $f \geq 0$, by the results of \cite{bbm} there exists a nonnegative function $u$, solution in $\huz \cap \elle\infty$ of
$$
\begin{cases}
-\Delta u + \overline{h}(u)|\D u|^{2} = \la\,f & \mbox{in $\Omega$,}\\
\hfill u = 0 \hfill & \mbox{on $\partial\Omega$.}
\end{cases}
$$
Clearly, $v$ is a solution of the above problem with $\mu$ instead of $\la$. We now follow the lines of \cite{as}: let $k > 0$ be fixed, and choose ${\rm e}^{-\overline{H}(u)}\,(\psi(u)-\psi(v))^{+}$ as test function in the equation for $u$, and ${\rm e}^{-\overline{H}(v)}\,(\psi(u)-\psi(v))^{+}$ in the equation for $v$, where (as in \rife{hpsi} with $\alpha=1$)
$$
\overline{H}(s) = \int_{0}^{s}\,\overline{h}(t)\,dt\,,
\quad
\mbox{and}
\quad
\psi(s) = \int_{0}^{s}\,{\rm e}^{-\overline{H}(t)}\,dt.
$$
We obtain
\be\label{cancella}
\begin{array}{l}
\dys
\io \D u \cdot \D (\psi(u)-\psi(v))^{+}\,{\rm e}^{-\overline{H}(u)}
\\
\dys
\qquad
-
\io \overline{h}(u)|\D u|^{2} (\psi(u)-\psi(v))^{+}\,{\rm e}^{-\overline{H}(u)}
\\
\dys
\qquad
+
\io \overline{h}(u)|\D u|^{2} (\psi(u)-\psi(v))^{+}\,{\rm e}^{-\overline{H}(u)}
\\
\dys
\quad
=
\la\io f\,{\rm e}^{-\overline{H}(u)}\,(\psi(u)-\psi(v))^{+}\,,
\end{array}
\ee
which can be rewritten as
$$
\io \D \psi(u) \cdot \D(\psi(u)-\psi(v))^{+}
=
\la\io f\,{\rm e}^{-\overline{H}(u)}\,(\psi(u)-\psi(v))^{+}\,.
$$
Analogously, we obtain
$$
\io \D \psi(v) \cdot \D(\psi(u)-\psi(v))^{+}
=
\mu\io f\,{\rm e}^{-\overline{H}(v)}\,(\psi(u)-\psi(v))^{+}\,.
$$
Subtracting the two identities, and recalling that $\la < \mu$, we obtain
$$
\io |\D(\psi(u)-\psi(v))^{+}|^{2}
\leq
\la\io \,f(\psi(u)-\psi(v))^{+}({\rm e}^{-\overline{H}(u)}-{\rm e}^{-\overline{H}(v)}).
$$
Since $\overline{H}$ and $\psi$ are increasing, the right hand side is negative, and so
$$
\io |\D(\psi(u)-\psi(v))^{+}|^{2}
= 0\,,
$$
which implies $u \leq v$. Since $u \leq v \leq \si - \eps$, we have $\overline{h}(u) = h(u)$ and so $u$ is a solution of \rife{mainla}.

To prove uniqueness of solutions for fixed $\la$, let $u$ and $v$ be solutions of \rife{mainla} for the same $\la$. If $H$ and $\psi$ are as in \rife{hpsi} (with $\al = 1$), using ${\rm e}^{-H(u)}(\psi(u)-\psi(v))^{+}$ and ${\rm e}^{-H(v)}(\psi(u)-\psi(v))^{+}$ as test functions, and reasoning as above, one proves that $u \leq v$; exchanging the roles of $u$ and $v$ yields the reverse inequality, so that $u = v$. To prove that if $\la < \mu$, then $u_{\la} \leq u_{\mu}$, use ${\rm e}^{-H(u_{\la})}(\psi(u_{\la})-\psi(u_{\mu}))^{+}$ and ${\rm e}^{-H(u_{\mu})}(\psi(u_{\la})-\psi(u_{\mu}))^{+}$ as test functions, and reason as above.
\end{proof}

\begin{remark}\label{compar}\rm
We want to stress that the fact that $M(x) \equiv I$ has been used in \rife{cancella} in order to cancel two equal terms: this is the key ingredient in the proof of the comparison principle for solutions of \rife{mainla}. In the general case of a symmetric matrix $M$, comparison results are few and partial: we refer the reader to the papers \cite{as} and \cite{bamu}. 
\end{remark}

\section{Nonexistence of solutions}

In this section we are going to prove Theorem \ref{nnhl1} and Theorem \ref{nhnl1}. We recall that we are dealing with solutions of
\begin{equation}\label{mainlan}
\bc
-\dive ( M (x) \D u) + h(u)|\D u|^2 =\la\, f \quad& \mbox{in $\Omega$,}\\
\hfill u=0 \hfill & \mbox{on $\partial\Omega$.}
\ec
\end{equation}
We begin with the proof of Theorem \ref{nnhl1}, i.e., with the case of $h$ in $L^{1}((0,\si))$.

\begin{proof}[{\sl Proof of Theorem \ref{nnhl1}}] Suppose by contradiction that there exists a solution $u \in \huz$ such that $0\leq u<\sigma $ a.e., and choose ${\rm e}^{-H (u)}\vp_{1}(f)$ as test function in the weak formulation of \rife{mainlan}, where, as in \rife{hpsi},
$$
H(s) = \frac{1}{\alpha}\,\int_{0}^{s}\,h(t)\,dt\,.
$$
We obtain
\be \label{albac}
\begin{array}{l}
\dys
\io M(x)\D u \cdot \D \vp_{1}(f) {\rm e}^{-H (u)}
\\
\dys
\qquad
-
\io M(x) \D u \cdot \D u\, \frac{{\rm e}^{-H (u)}}{\alpha} h (u) \vp_{1}(f)
\\
\qquad
\dys
+ \io |\D u |^2 {\rm e}^{-H (u)} h (u) \vp_{1}(f) 
\\
\dys
\quad
=
\la \io f\, {\rm e}^{-H (u) } \vp_{1}(f) \,.
\end{array}
\ee
Thus, if we define (again as in \rife{hpsi})
$$
\psi (s)= \int^s_0 {\rm e}^{-H (t)}dt\,,
$$
using the definition of $\la_1(f)$ and $\vp_1(f)$, as well as the symmetry of $M$, we deduce that the first term in \rife{albac} can be written as 
$$
\begin{array}{r@{\hspace{2pt}}c@{\hspace{2pt}}l}
\dys
\io M (x) \D u \cdot \D \vp_{1}(f) {\rm e}^{-H (u)}
& = &
\dys
\io M (x) \D \vp_{1}(f) \cdot \D \psi(u)
\\
& = &
\dys
\la_1(f) \io f \vp_{1}(f) \psi(u)\,.
\end{array}
$$
Therefore, using \rife{a1} and dropping nonnegative terms, we have, from \rife{albac}, 
$$
\io f\, \vp_{1}(f) [\la_1(f) \psi(u) - \la {\rm e}^{-H (u) }]\geq 0
$$ 
Consider now the function $\Theta : [0, \si)\to \re^+$ defined by
$$
\Theta (s)=\la_1(f) \psi(s)- \la {\rm e}^{-H (s) } \,.
$$
Since both $\psi(0)=0$ and ${\rm e}^{-H (0) }=1$, we have $\Theta (0)= -\la <0$. 
Moreover,
$$
\Theta'(s) = \left(\la_1(f) + \la\frac{h(s)}{\alpha}\right)\,{\rm e}^{-H(s)} \geq 0\,,
$$
so that
$$
\Theta (s) \leq \Theta (\si) = \la_1(f)\psi(\si) -\la {\rm e}^{-H (\sigma) }
$$
Therefore, if $\la >\la_1(f) {\rm e}^{H(\sigma)}\psi(\si)$ we have 
$$
0\leq \io f \,\vp_{1}(f)[\la_1(f) \psi(u) - \la {\rm e}^{-H (u) }] < 0 \,,
$$
which is a contradiction.
\end{proof}

We turn now to the study of \rife{mainlan} under assumption \rife{assumpt2}, which implies that $\sqrt{h}$ belongs to $L^{1}((0,\si))$, while $h$ itself does not. Since $M \equiv I$, we are going to deal with solutions of
\be\label{noexis}
\bc
-\Delta u + h(u)|\D u|^{2} = \la\,f & \mbox{in $\Omega$,} \\
\hfill u = 0 \hfill & \mbox{on $\partial\Omega$.}
\ec
\ee
As in \rife{hpsi}, we define
$$
H(s) = \int_{0}^{s}\,h(t)\,dt\,,
\quad
\mbox{and}
\quad
\psi(s) = \int_{0}^{s}\,{\rm e}^{-H(t)}\,dt\,.
$$
By the assumptions on $h$, $H$ is unbounded on $(0,\si)$, while $\psi$ is bounded. We define $L = \psi(\si)$, so that $0 \leq \psi(s) \leq L$, and $\psi$ will be increasing (hence invertible) from $[0,\si]$ to $[0,L]$.

Let now $u\in\huz$ be a solution of \rife{noexis}, with $0 \leq u < \si$ almost everywhere in $\Omega$. Defining $v = \psi(u)$, we have
$$
\D v = {\rm e}^{-H(u)}\D u\,,
\quad
\Delta v = {\rm e}^{-H(u)}[\Delta u - h(u)|\D u|^{2}]\,.
$$
If we set
\be\label{defg}
g(s) = {\rm e}^{-H(\psi^{-1}(s))}\,,
\ee
we have that $v\in \huz$ is a solution of
\be\label{slin}
\bc
-\Delta v = \la\,f\,g(v) & \mbox{in $\Omega$,} \\
\hfill v = 0 \hfill & \mbox{on $\partial\Omega$.}
\ec
\ee

\begin{lemma}\label{ong}\sl
The function $g: [0,L] \to \re$ defined by \rife{defg} is such that:
\begin{itemize}
\item[i)] $g(0) = 1$ and $g(L) = 0$;

\item[ii)] $g$ is decreasing;

\item[iii)] $\dys\lim_{s \to L^{-}}\,g'(s) = -\infty$;

\item[iv)] 
$$
\int_{0}^{L}\,\frac{dt}{\sqrt{\int_{t}^{L}\,g(s)\,ds}} < +\infty\,.
$$
\end{itemize} 
\end{lemma}

\begin{example}\label{partic}\rm
In the particular case $h(s) = \frac{A}{1-s}$ (with $A > 0)$, it is easy to see that 
$$
g(s) = [1 - (1+A)s]^{\frac{A}{1+A}}\,.
$$
\end{example}

\begin{proof}[{\sl Proof}] Since $\psi^{-1}(0) = 0$, we have $g(0) = 1$. On the other hand, since both $\psi^{-1}(L) 
= \si$ and $H$ is unbounded, we have $g(L) = 0$, so that i) is proved. Furthermore, using that
$$
(\psi^{-1}(s))' = {\rm e}^{H(\psi^{-1}(s))}\,,
$$
we have
\be\label{derig}
g'(s) = -h(\psi^{-1}(s))\,{\rm e}^{-H(\psi^{-1}(s))}(\psi^{-1}(s))' = -h(\psi^{-1}(s))\,,
\ee
which implies both ii) and iii). As for iv), it is easy to see, using the definition of $g$, and by changing variables, that
$$
\int_{0}^{L}\,\frac{dt}{\sqrt{\int_{t}^{L}\,g(s)\,ds}} < +\infty
\quad
\iff
\quad
\int_{0}^{\si}\,\frac{{\rm e}^{-H(t)}\,dt}{\sqrt{\int_{t}^{\si}\,{\rm e}^{-2H(s)}\,ds}} < +\infty\,.
$$
We are going to prove that there exists $D > 0$ such that
\be\label{hop}
\lim_{t \to \si^{-}} \frac{h(t)\,\int_{t}^{\si}\,{\rm e}^{-2H(s)}\,ds}{{\rm e}^{-2H(t)}} = D\,.
\ee
Once we prove this, we will have that
$$
\int_{0}^{\si}\,\frac{{\rm e}^{-H(t)}\,dt}{\sqrt{\int_{t}^{\si}\,{\rm e}^{-2H(s)}\,ds}} < +\infty
\quad
\iff
\quad
\int_{0}^{\si}\,\sqrt{h(t)}\,dt < +\infty\,,
$$
and the latter result is true by \rife{assumpt2} since $1 \leq \ga < 2$.
By \rife{assumpt2}, we have
$$
\lim_{t \to \si^{-}} \frac{h(t)\,\int_{t}^{\si}\,{\rm e}^{-2H(s)}\,ds}{{\rm e}^{-2H(t)}} = 
C \,
\lim_{t \to \si^{-}} \frac{\int_{t}^{\si}\,{\rm e}^{-2H(s)}\,ds}{(\si - t)^{\gamma}\,{\rm e}^{-2H(t)}}\,.
$$
Applying the de l'H\^{o}pital rule, we have
$$
\lim_{t \to \si^{-}} \frac{\int_{t}^{\si}\,{\rm e}^{-2H(s)}\,ds}{(\si - t)^{\gamma}\,{\rm e}^{-2H(t)}}
=
\lim_{t \to \si^{-}} \frac{-{\rm e}^{-2H(t)}}{-[\ga(\si -t)^{\ga-1} + 2(\si-t)^{\ga}h(t)]{\rm e}^{-2H(t)}}\,,
$$
and the latter limit is equal (again by \rife{assumpt2}) to $\frac{1}{2C}$ (if $\ga > 1$) or $\frac{1}{2C+1}$ (if $\ga = 1$). Therefore, \rife{hop} holds, and iv) is proved.
\end{proof}

Since solutions $v$ of \rife{slin} which are almost everywhere smaller than $L$ correspond to solutions $u$ of \rife{noexis} which are almost everywhere smaller than $\si$ {\sl via} the transformation $u = \psi^{-1}(v)$, we are now going to consider problem \rife{slin} on its own.

\begin{theo}\label{exisslin}\sl
For every $\la > 0$ there exists a unique solution $v_{\la}$ of \rife{slin}, with $v_{\la}$ in $\huz$, and $0 \leq v_{\la} \leq L$.
\end{theo}

\begin{proof}[{\sl Proof}]
Define
$$
\ol{g}(s) = 
\bc
\hfill 1 \hfill & \mbox{if $s < 0$,} \\
g(s) & \mbox{if $0 \leq s \leq L$,} \\
\hfill 0 \hfill & \mbox{if $s > L$,}
\ec
$$
and consider the problem
$$
\bc
-\Delta w = \la\,f\,\ol{g}(w) & \mbox{in $\Omega$,} \\
\hfill w = 0 \hfill & \mbox{on $\partial\Omega$.}
\ec
$$
It is easy to see, by using a fixed point argument, that for every $\la > 0$ there exists a solution $w$ in $\huz$, which is nonnegative since $\la\,f\,\ol{g}(w) \geq 0$ in $\Omega$. Taking $(w-L)^{+}$ as test function in the weak formulation, and using the definition of $\ol{g}$ as well as the fact that $f \geq 0$, we have $w \leq L$. Hence, by definition of $\ol{g}$, $\ol{g}(w) = g(w)$, and so $w$ is a solution of \rife{slin}. Uniqueness of solutions then follows from the next lemma.
\end{proof}

\begin{lemma}\label{subsup}\sl
Let $\ul{v}$ and $\ol{v}$ in $H^1 (\Omega)$ be such that
$$
-\Delta \ul{v} \leq \la\,f\,g(\ul{v})\,,
\quad
-\Delta \ol{v} \geq \la\,f\,g(\ol{v})\,,
$$
in $H^{-1}(\Omega)$ and $(\ul{v}-\ol{v}) \leq 0$ on $\partial \Omega$. Then $\ul{v} \leq \ol{v}$.
\end{lemma}

\begin{proof}[{\sl Proof}]
Subtracting the above inequalities and choosing $(\ul{v}-\ol{v})^{+}$ as test function yields
$$
0
\leq
\io |\D(\ul{v}-\ol{v})^{+}|^{2}
\leq
\la\,\io \,f\,(g(\ul{v}) - g(\ol{v}))(\ul{v}-\ol{v})^{+}
\leq 0\,,
$$
since, by ii) of Lemma \ref{ong}, $g$ is decreasing. Therefore, $(\ul{v}-\ol{v})^{+} = 0$, i.e., $\ul{v} \leq \ol{v}$.
\end{proof}

\subsection{Construction of unidimensional ``flat'' solutions}

To prove that \rife{noexis} has no solutions for $\la$ large enough, we are going to prove that \rife{slin} has solutions $v$ such that $\mis(\{v = L\}) > 0$ if $\la$ is large enough. In order to deal with the general $N$-dimensional case, we are going to study the one-dimensional equation first, with $f \equiv 1$. Therefore, we are going to fix $R > 0$ and consider the problem
\be\label{1d}
\bc
-v''(s) = \la\,g(v(s)) & \mbox{in $(-R,R)$,} \\
\hfill v(\pm R) = 0\,. \hfill 
\ec
\ee
We know from Theorem \ref{exisslin} that a solution $v$ of \rife{1d} exists for every $\la$ and for every $R$, with $0 \leq v \leq L$. In order to study the properties of the solutions $v$ as $\la$ changes, we are going to study the solutions of a ``shooting'' problem. Define $g(s) \equiv 1$ for $s < 0$, let $0 < \ell < L$, and consider the solution $\vll$ of the Cauchy problem
$$
\bc
-\vll''(s) = \la\,g(\vll(s)) & \mbox{for $s \geq 0$,} \\
\hfill \vll(0) = \ell\,,\ \vll'(0) = 0\,. \hfill 
\ec
\leqno{(P_{\ell})}
$$
Since $g$ is Lipschitz continuous around $\ell$, then there exists a unique solution $\vll$ of \pll\ (at least locally near~0). Note that ($P_{L}$) is singular, since $g'(L) = -\infty$ by iii) of Lemma \ref{ong}. Thus, ($P_{L}$) has a maximal solution, which is $v^{L}(s) \equiv L$, and a minimal solution $v_{L}(s)$, such that any other solution $v$ satisfies $v_{L}(s) \leq v(s) \leq v^{L}(s)$. We are going to prove that ($P_{L}$) has infinitely many solutions not identically equal to $L$.

\begin{theo}\label{pill}\sl
Let $\vll$ be the solution of \pll. Then:
\begin{itemize}
\item[i)] $\vll'(s) < 0$ for every $s > 0$, and $\vll$ is uniquely defined for every $s \geq 0$;

\item[ii)] for every $\ell < L$ there exists $\rll > 0$ such that $\vll(\rll) = 0$. Furthermore,
\be\label{ralph}
\rll = \frac{1}{\sqrt{2\la}}\,\int_{0}^{\ell}\,\frac{dt}{\sqrt{\int_{t}^{\ell}\,g(s)\,ds}};
\ee

\item[iii)] as $\ell$ increases to $L$, $\vll$ increases; therefore, $\rll$ increases, and converges to the finite value
$$
R_{L} = \frac{1}{\sqrt{2\la}}\,\int_{0}^{L}\,\frac{dt}{\sqrt{\int_{t}^{L}\,g(s)\,ds}};
$$

\item[iv)] as $\ell$ increases to $L$, $\vll$ uniformly converges on $[0,\ol{s}]$, for every $\ol{s} > 0$, to $v_{L}$, the minimal solution of ($P_{L}$), with $v_{L}(s) < L$ for every $s > 0$, and $v_{L}(R_{L}) = 0$. Furthermore, $v_{L}$ is the only solution of ($P_{L}$) such that $v_{L}(s) < L$ for every $s > 0$;

\item[v)] if $w$ is a solution (not identically equal to $L$) of ($P_{L}$), then
\be\label{solpiatte}
w(s) = 
\bc
\hfill L \hfill & \mbox{if $0 \leq s \leq R$,}\\
v_{L}(s-R) & \mbox{if $s > R$,}
\ec
\ee
for some $R > 0$.
\end{itemize}
\end{theo}

\begin{proof}[{\sl Proof}]
i) The proof is easy: since $\vll''(0) = -\la g(\ell) < 0$, we have $\vll'(s) < 0$ for $s > 0$ and sufficiently small. If there exists $\ol{s}$ such that $\vll'(\ol{s}) = 0$, then $\vll$ is decreasing in $[0,\ol{s}]$, so that (since $g$ is decreasing as well), $g(\vll(s)) \geq g(\ell)$ for every $s$ in $[0,\ol{s}]$. But then we get a contradiction integrating the equation:
$$
0 = \vll'(\ol{s}) - \vll'(0) = \int_{0}^{\ol{s}} \vll''(t) dt
= -\la\int_{0}^{\ol{s}} g(\vll(t)) dt \leq -g(\ell) \ol{s} < 0\,.
$$
Since $\vll'(s) < 0$, $\vll$ stays away from $L$ (where $g$ is no longer Lipschitz continuous). Therefore, $\vll$ is unique and exists for every $s \geq 0$.

\medskip

ii) If we suppose that $\vll(s) \neq 0$ for every $s \geq 0$, then $\vll(s) > 0$ for every $s \geq 0$. Since $\vll$ is decreasing by i), there exists
\be\label{limt}
\lim_{s \to +\infty}\,\vll(s) = T\,,
\ee
with $0 \leq T < \ell$. Therefore, from the equation we get that
$$
\lim_{s \to +\infty}\,\vll''(s) = -\la\,g(T) < 0\,,
$$
which implies
$$
\lim_{s \to +\infty}\,\vll'(s) = \lim_{s \to +\infty}\,\int_{0}^{s}\,\vll''(t)\,dt = -\infty\,.
$$
This, however, yields
$$
\lim_{s \to +\infty}\,\vll(s) = \lim_{s \to +\infty}\,\int_{0}^{s}\,\vll'(t)\,dt + \ell = -\infty\,,
$$
a contradiction with \rife{limt}. Thus, there exists $\rll > 0$ such that $\vll(\rll) = 0$.

We now multiply the equation by $\vll'$ and integrate on $[0,s]$. We obtain, using the initial conditions on $\vll$ and $\vll'$,
$$
-\frac12 [\vll'(s)]^{2} = \la\int_{0}^{s}\,g(\vll(t))\,\vll'(t)\,dt =
\la [G(\vll(s)) - G(\ell)]\,,
$$
where we have defined
$$
G(s) = \int_{0}^{s}\,g(t)\,dt\,.
$$
Therefore, taking into account that $\vll'(s) < 0$,
\be\label{derv}
\vll'(s) = -\sqrt{2\la [G(\ell) - G(\vll(s))]}\,.
\ee
Since $\vll'(s) < 0$ for every $s > 0$, we have $G(\vll(s)) \neq G(\ell)$ for every $s > 0$. Therefore, we can divide by $\sqrt{G(\ell) - G(\vll(s))}$ and integrate between~0 and $\rll$ to obtain
$$
\int_{0}^{\rll}\,\frac{\vll'(s)\,ds}{\sqrt{G(\ell) - G(\vll(s))}} = -\sqrt{2\la}\,\rll\,.
$$
Setting $t = \vll(s)$ in the first integral, recalling that $\vll(0) = \ell$, that $\vll(\rll) = 0$, and the definition of $G$, we have
$$
\int_{0}^{\ell}\,\frac{dt}{\sqrt{\int_{t}^{\ell}\,g(s)\,ds}} = \sqrt{2\la}\,\rll\,,
$$
which then gives \rife{ralph}. Observe that the integral is singular at $t = \ell$, but that, since $g$ is decreasing,
$$
\int_{0}^{\ell}\,\frac{dt}{\sqrt{\int_{t}^{\ell}\,g(s)\,ds}}
\leq
\frac{1}{\sqrt{g(\ell)}}
\int_{0}^{\ell}\,\frac{dt}{\sqrt{\ell - t}}
=
2\sqrt{\frac{\ell}{g(\ell)}}
\,.
$$

\medskip

iii) To prove that $\vll$ increases as $\ell$ increases, let $\ell < \ell'$ and let $\vll$ and $v_{\ell'}$ be the solutions of \pll\ and ($P_{\ell'}$) respectively. Since $\vll(0) < v_{\ell'}(0)$, then $\vll \leq v_{\ell'}$ in a neighbourhood of the origin. If there exists $\ol{s}$ such that $\vll(\ol{s}) = v_{\ell'}(\ol{s})$ then, subtracting the equations, multiplying by $\vll - v_{\ell'}$, and integrating on $[0,\ol{s}]$ yields
$$
-\int_{0}^{\ol{s}}(\vll''- v_{\ell'}'')(\vll - v_{\ell'})
=
\la\int_{0}^{\ol{s}}(g(\vll) - g(v_{\ell'}))(\vll- v_{\ell'})\,,
$$
and the right hand side is negative since $g$ is decreasing. Integrating by parts the left hand side (and observing that $\vll'(0)- v_{\ell'}'(0) = 0$, and that $\vll(\ol{s})- v_{\ell'}(\ol{s}) = 0$), we obtain
$$
0 \leq \int_{0}^{\ol{s}}(\vll'(t) - v_{\ell'}'(t))^{2}\,dt \leq 0\,,
$$
which implies $\vll'(s) \equiv v_{\ell'}'(s)$ on $[0,\ol{s}]$, and so 
$$
\vll(\ol{s}) = \ell + \int_{0}^{\ol{s}}\,\vll'(t)\,dt
= \ell + \int_{0}^{\ol{s}}\,v_{\ell'}'(t)\,dt = \ell - \ell' + v_{\ell'}(\ol{s})\,,
$$
a contradiction with the assumptions $\vll(\ol{s}) = v_{\ell'}(\ol{s})$ and $\ell < \ell'$. Thus, $\vll$ increases with $\ell$, which implies that $\rll$ increases as well; therefore, there exists
$$
R_{L} = \lim_{\ell \to L^{-}}\,\rll\,.
$$
Using \rife{ralph}, we have
$$
\rll = \frac{1}{\sqrt{2\la}}\int_{0}^{\ell} \frac{dt}{\sqrt{\int_{t}^{\ell} g(s) ds}}
=
\frac{1}{\sqrt{2\la}}\int_{0}^{L} \frac{\chi_{[0,\ell)}(t) dt}{\sqrt{\int_{t}^{\ell} g(s) ds}} = \frac{1}{\sqrt{2\la}}\int_{0}^{L}\th_{\ell}(t) dt,
$$
where we have defined
$$
\th_{\ell}(t) = \frac{\chi_{[0,\ell)}(t)}{\sqrt{\int_{t}^{\ell}\,g(s)\,ds}}\,.
$$
We clearly have, for every $0 \leq t < L$,
$$
\lim_{\ell \to L^{-}}\,\th_{\ell}(t) = \frac{\chi_{[0,L)}(t)}{\sqrt{\int_{t}^{L}\,g(s)\,ds}} = \th_{L}(t)\,,
$$
so that the proof of iii) will be complete if we prove that
\be\label{limleb1}
\lim_{\ell \to L^{-}}\,\int_{0}^{L}\,\th_{\ell}(t)\,dt = \int_{0}^{L}\,\th_{L}(t)\,dt\,.
\ee
Define $L_{n} = L(1 - \frac{1}{2^{n}})$, and 
$$
\th_{n}(t) = \th_{L_{n}}(t) = \frac{\chi_{[0,L_{n})}(t)}{\sqrt{\int_{t}^{L_{n}}\,g(s)\,ds}}\,.
$$
If $m < n$, in $[0,L_{m})$ we have, since $g(s) \geq 0$,
$$
\th_{m}(t) = \frac{1}{\sqrt{\int_{t}^{L_{m}}\,g(s)\,ds}}
\geq
\frac{1}{\sqrt{\int_{t}^{L_{n}}\,g(s)\,ds}} = \th_{n}(t)\,,
$$
so that
\be\label{mn}
m < n \ \Rightarrow \ \th_{m}(t) \geq \th_{n}(t) \quad \mbox{in $[L_{m-1},L_{m})$}\,.
\ee
Let now
$$
\th(t) = \sum_{k = 0}^{+\infty}\,\th_{k+1}(t)\,\chi_{[L_{k},L_{k+1})}(t)\,,
$$
and observe that, by \rife{mn}, 
$$
\th(t) \geq \sum_{k = 0}^{n-1}\,\th_{k+1}(t)\chi_{[L_{k},L_{k+1})}(t) \geq \sum_{k = 0}^{n-1}\,\th_{n}(t)\chi_{[L_{k},L_{k+1})}(t) = \th_{n}(t)\,.
$$
Since $\th_{n}$ converges to $\th_{L}$, we can apply Lebesgue theorem to prove \rife{limleb1} if we prove that
\be\label{limleb2}
\int_{0}^{L}\,\th(t)\,dt < +\infty\,.
\ee
We have, by monotone convergence theorem, and recalling the definition of $\th_{k}(s)$,
$$
\int_{0}^{L}\,\th(t)\,dt 
=
\sum_{k = 0}^{+\infty}\,\int_{L_{k}}^{L_{k+1}}\,\th_{k+1}(t)\,dt
=
\sum_{k = 0}^{+\infty}\,\int_{L_{k}}^{L_{k+1}}\,\frac{dt}{\sqrt{\int_{t}^{L_{k+1}}\,g(s)\,ds}}\,.
$$
Since $g$ is decreasing, we have
$$
\int_{t}^{L_{k+1}}\,g(s)\,ds \geq g(L_{k+1})(L_{k+1}-t)\,,
$$
so that
$$
\int_{0}^{L}\,\th(t)\,dt
\leq
\sum_{k = 0}^{+\infty}\,
\frac{1}{\sqrt{g(L_{k+1})}}\,\int_{L_{k}}^{L_{k+1}}\,\frac{dt}{\sqrt{L_{k+1}-t}}
=
2\sum_{k = 0}^{+\infty}\,\frac{\sqrt{L_{k+1} - L_{k}}}{\sqrt{g(L_{k+1})}}\,.
$$
Recalling the definition of $L_{k}$, we then have
\be\label{primastima}
\int_{0}^{L}\,\th(t)\,dt \leq 2\sqrt{L}\,\sum_{k = 0}^{+\infty}\,\frac{1}{\sqrt{g(L_{k+1})2^{k+1}}}
\,.
\ee
On the other hand, we have
$$
\int_{0}^{L}\,\frac{dt}{\sqrt{\int_{t}^{L}\,g(s)\,ds}}
=
\sum_{k = 0}^{+\infty}\,\int_{L_{k}}^{L_{k+1}}\,\frac{dt}{\sqrt{\int_{t}^{L}\,g(s)\,ds}}\,.
$$
Once again, since $g$ is decreasing, we have, for $t$ in $(L_{k},L_{k+1})$,
$$
\int_{t}^{L}\,g(s)\,ds \leq g(t)\,(L-t) \leq g(L_{k})(L-t)\,,
$$
and so
$$
\int_{0}^{L}\,\frac{dt}{\sqrt{\int_{t}^{L}\,g(s)\,ds}}
\geq
\sum_{k = 0}^{+\infty}\,\frac{1}{\sqrt{g(L_{k})}}\int_{L_{k}}^{L_{k+1}}\,\frac{dt}{\sqrt{L-t}}\,.
$$
Therefore, 
$$
\int_{0}^{L}\,\frac{dt}{\sqrt{\int_{t}^{L}\,g(s)\,ds}}
\geq
2\sum_{k = 0}^{+\infty}\,\frac{\sqrt{L-L_{k}}-\sqrt{L-L_{k+1}}}{\sqrt{g(L_{k})}}\,,
$$
and so, recalling the definition of $L_{k}$,
$$
\int_{0}^{L}\,\frac{dt}{\sqrt{\int_{t}^{L}\,g(s)\,ds}}
\geq
\sqrt{2L}(\sqrt{2}-1)\,\sum_{k = 0}^{+\infty}\,\frac{1}{\sqrt{g(L_{k})2^{k}}}\,.
$$
Recalling \rife{primastima}, we then have
$$
\int_{0}^{L}\,\th(t)\,dt
\leq
\frac{\sqrt{2}}{\sqrt{2}-1}\,\int_{0}^{L}\,\frac{dt}{\sqrt{\int_{t}^{L}\,g(s)\,ds}}\,,
$$
and the latter integral is finite by iv) of Lemma \ref{ong}. Thus \rife{limleb2} is proved, and the proof of iii) is complete.

\medskip

iv) Let $\ol{s} > 0$ be fixed, and consider $\vll(s)$ on $[0,\ol{s}]$. Since $\vll$ increases as $\ell$ increases to $L$, then $\vll(s)$ converges to some function $v_{L}(s)$.
Since $g(s) \leq 1$, then
$$
|\vll'(s)| = \la\,\int_{0}^{s}\,g(\vll(t))\,dt \leq \la\,s \leq \la\,\ol{s}\,,
\quad
\forall s \in [0,\ol{s}]\,,
$$
so that $|\vll'(s)|$ is uniformly bounded with respect to $\ell$ on $[0,\ol{s}]$. Since $\vll(0) = \ell$ is bounded as well, Ascoli-Arzel\`{a} theorem implies that $\vll$ converges uniformly to $v_{L}$ as $\ell$ tends to $L$. This convergence, and the fact that $\rll$ converges to $R_{L}$, imply that $v_{L}(R_{L}) = 0$.
Since $\vll$ increases to $v_{L}$, the fact that $g$ is decreasing implies that $g(\vll(s))$ decreases to $g(v_{L}(s))$. Therefore, and since $g(s) \leq 1$, Lebesgue theorem implies, for every $s$ in $[0,\ol{s}]$,
$$
\lim_{\ell \to L^{-}}\,\int_{0}^{s}\,g(\vll(t))\,dt 
=
\int_{0}^{s}\,g(v_{L}(t))\,dt\,,
$$
so that for every $s$ in $[0,\ol{s}]$ there exists
$$
\lim_{\ell \to L^{-}}\,\vll'(s) = 
- \lim_{\ell \to L^{-}}\,\la\int_{0}^{s}\,g(\vll(t))\,dt
=
-\la\int_{0}^{s}\,g(v_{L}(t))\,dt = w(s)\,.
$$
Since from the equation we have that $(\vll')'$ is uniformly bounded (with respect to $\ell$), and furthermore $\vll'(0) = 0$ is bounded as well, a further application of Ascoli-Arzel\`a theorem implies that $\vll'$ uniformly converges to $w$ on $[0,\ol{s}]$. This fact (together with the convergence of $\vll$ to $v_{L}$) implies that $w = v_{L}'$. Therefore,
$$
v_{L}'(s) = -\la\int_{0}^{s}\,g(v_{L}(t))\,dt\,,
$$
which implies (since $g(v_{L}(t))$ is continuous) that $v_{L}$ is a solution of
$$
\bc
-v_{L}''(s) = \la\,g(v_{L}(s)) & \mbox{for $s \geq 0$,}\\
\hfill v_{L}(0) = L\,,\ v_{L}'(0) = 0\,. \hfill
\ec
\leqno{(P_{L})}
$$
Since $v_{L}(R_{L}) = 0$, we have that $v_{L}$ is not identically equal to $L$. To prove that $v_{L}$ is the minimal solution of ($P_{L}$), let $w$ be a solution of ($P_{L}$). We then have that $w \geq \vll$ for every $\ell < L$. The proof of this fact can be achieved with the same ideas used to prove that $\vll \leq v_{\ell'}$ if $\ell < \ell'$ in iii). Since $v_{L}$ is the limit of $\vll$, we then have $v_{L} \leq w$, as desired.

To prove that $v_{L}(s) < L$ for every $s > 0$, suppose that there exists $\ol{s} > 0$ such that $v_{L}(\ol{s}) = L$. Then $v_{L}(s) \equiv L$ in $[0,\ol{s}]$. Indeed, since $v_{L}''(s) \leq 0$, $v_{L}$ is concave, and so its graph is above the (horizontal) line connecting $(0,L)$ and $(\ol{s},L)$, and below the tangent line at $(0,L)$, which (since $v_{L}'(0) = 0$) is the same horizontal line. Thus, $v_{L}$ is constantly equal to $L$ on $[0,\ol{s}]$. If we define
$$
S = \sup \{ s > 0 : v_{L}(s) = L\}\,,
$$
we have that $0 < \ol{s} \leq S < R_{L}$, that $v_{L}(S) = L$, and that $v_{L}'(S) = 0$. Therefore, $v_{L}$ is a solution of
$$
\bc
-v_{L}''(s) = \la g(v_{L}(s)) & \mbox{for $s \geq S$,}\\
\hfill v_{L}(S) = L\,,\ v_{L}'(S) = 0\,,
\ec
$$
which is decreasing (hence strictly decreasing) for $s > S$ (the proof of this fact is analogous to the one in i)).
If we define $z(s) = v_{L}(s+S)$, then $z(s) \leq v_{L}(s)$, and $z$ is a solution of ($P_{L}$). Since $v_{L}$ is the minimal solution of the same problem, then $v_{L}(s) \leq z(s)$, so that $v_{L} \equiv z$. This implies that $S = 0$, a contradiction since $S \geq \ol{s} > 0$.

To prove that $v_{L}$ is the only solution of $P_{L}$ such that $v_{L}(s) < L$ for every $s > 0$, let $w$ be another such solution. Since $w$ has no ``flat'' zones (being $w'(s) < 0$ for every $s > 0$), then $w(R_{L}) = 0$ (just start from the equation and perform the same calculations used to obtain \rife{ralph} for $\rll$). We therefore have that $v_{L}(0) = w(0)$, and that $v_{L}(R_{L}) = w(R_{L})$. Subtracting the equations satisfied by $v_{L}$ and $w$, multiplying by $v_{L}- w$, integrating, using that $g$ is decreasing and that $v_{L}(s) \leq w(s)$ for every $s > 0$ since $v_{L}$ is the minimal solution, we have $v_{L}' \equiv w'$, and so $v_{L} \equiv w$.

\medskip

v) Let $w$ be a solution of ($P_{L}$) which is not identically equal to $L$. Since $v_{L}(s)$ is the unique solution of ($P_{L}$) which is different from $L$ for every $s > 0$, there exists $\ul{s}$ such that $w(\ul{s}) = L$. Reasoning as in iv) (i.e., using the concavity of $w$), we have that $w(s) \equiv L$ on $[0,\ul{s}]$. Define 
$$
R = \sup\{s > 0 : w(s) = L \} > 0\,,
$$
and observe that $R < +\infty$ since $w$ is not identically equal to $L$. Setting $z(s) = w(s+R)$ for $s \geq 0$, and reasoning as in iv), one has that $z$ is a solution of ($P_{L}$), and that $z(s) < L$ for every $s > 0$. By uniqueness, $z(s) = v_{L}(s)$, and so $w$ is of the form \rife{solpiatte}.
\end{proof}

\medskip

Theorem \ref{pill} allows us to prove the existence of a flat solution in dimension $N=1$. In fact, let now $R$ be fixed, and consider the solution $v$ to the Dirichlet problem \rife{1d}. Since $v$ is symmetric with respect to the origin, then $v'(0) = 0$. Setting $\ell = v(0)$, we then have that $v = \vll$ (if $\ell < L$), or that $v$ is one of the solutions of ($P_{L}$) which are between $v_{L}$ and $v^{L} \equiv L$ (if $\ell = L$). 
Therefore, it is easy to see that $v$ is almost everywhere smaller than $L$ (i.e., $v$ has no ``flat'' zones where it is constantly equal to $L$) if and only if $R \leq R_{L}$. In other words, if $\la$ is fixed, and $R > R_{L}$, then any solution $v$ of \rife{1d} has a nonzero measure ``flat'' zone $\{v = L\}$. Observe that since 
$$
R_{L} = \frac{1}{\sqrt{2\la}}\int_{0}^{L}\,\frac{dt}{\int_{t}^{L}\,g(s)\,ds}\,,
$$
then $R_{L} = C\la^{-\frac12}$, so that, for every $R > 0$, we have $R > R_{L}$ for $\la$ large enough. Therefore, problem \rife{1d} has ``flat'' solutions; hence there are no solutions almost everywhere smaller than $\si$ of the corresponding quasilinear singular problem \rife{noexis} in dimension~1; i.e., $\La_{f \equiv 1}$ is finite.

\subsection{Construction of radial ``flat'' solutions}

Now we turn our attention to the $N$-dimensional semilinear problem \rife{slin}. We are going to prove that, under the same assumptions on $h$ for which the one-dimensional semilinear problem has ``flat'' solutions for $\la$ large, problem \rife{slin} has ``flat'' solutions as well. In order to do that, we are going to prove the following result.

\begin{theo}\label{sotto}\sl
Let $\Omega = B_{R}(0)$ and $f \equiv 1$; then, for $\la$ large enough, there exists a subsolution $\ul{v}$ of \rife{slin} such that $\mis(\{v = L\}) > 0$. 
\end{theo}

\begin{proof}
Let $v_{L}(s)$ be the minimal solution of the one-dimensional Cauchy problem
$$
\bc
-v''_{L}(s) = g(v_{L}(s)) & \mbox{if $s \geq 0$,}\\
\hfill v_{L}(0) = L\,, \ v'_{L}(0) = 0\,,
\ec
$$
and let $R_{L}$ be such that $v_{L}(R_{L}) = 0$.
Let $\ul{R} > 0$, and consider the function
$$
w(s) = 
\bc
\hfill L \hfill & \mbox{if $0 \leq s \leq \ul{R}$,} \\
v_{L}(s-\ul{R}) & \mbox{if $s > \ul{R}$.}
\ec
$$ 
By \rife{solpiatte}, $w$ is a $C^{2}$ ``flat'' solution of the same problem solved by $v_{L}$, with $w(R_{L}+\ul{R}) = 0$. 
If we consider $w$ as a radial function, we then have
$$
-\Delta w = -w''(r) - (N-1)\frac{w'(r)}{r} = g(w(r)) - (N-1)\frac{w'(r)}{r}\,.
$$
We are going to prove that there exists $\ul{R} > 0$ such that $-\Delta w \leq 2\,g(w)$ in $B_{R_{L} + \ul{R}}(0)$. In order to do that, it is enough to prove that
$$
-(N-1)\frac{w'(r)}{r} \leq g(w(r))\,,
$$
for every $\ul{R} \leq r \leq R_{L} + \ul{R}$, since if $r < \ul{R}$ we have $-\Delta w = 0 = 2g(w)$. Thus, we have to prove that
$$
-(N-1)\frac{v_{L}'(r-\ul{R})}{r} \leq g(v_{L}(r-\ul{R}))\,, \quad \forall 0 \leq r -\ul{R} \leq R_{L} \,, 
$$
or, equivalently, that
$$
-(N-1)\frac{v_{L}'(s)}{s + \ul{R}} \leq g(v_{L}(s))\,,
\quad
\forall 0 \leq s \leq R_{L}\,.
$$
Recalling \rife{derv}, written for $\ell = L$ and $\la = 1$, and the definition of $G$, we have
$$
v'_{L}(s) = -\sqrt{2\int_{v_{L}(s)}^{L}\,g(t)\,dt}\,,
$$
so that we have to prove that
$$
(N-1)\frac{\sqrt{2\int_{v_{L}(s)}^{L}\,g(t)\,dt}}{s + \ul{R}} \leq g(v_{L}(s))\,,
\quad
\forall 0 \leq s \leq R_{L}\,.
$$
We are going to prove that there exists $\ul{R} > 0$ such that the stronger inequality
$$
(N-1)\frac{\sqrt{2\int_{v_{L}(s)}^{L}\,g(t)\,dt}}{\ul{R}} \leq g(v_{L}(s))\,,
\quad
\forall 0 \leq s \leq R_{L}\,,
$$
holds; that is, $\ul{R} > 0$ is such that
$$
\int_{v_{L}(s)}^{L}\,g(t)\,dt
\leq
\frac{\ul{R}^{2}}{2(N-1)^{2}}\,g^{2}(v_{L}(s))\,,
\quad
\forall 0 \leq s \leq R_{L}\,.
$$
Since $v_{L}(s)$ takes values in $[0,L]$ as $s$ ranges between~0 and~$R_{L}$, finding $\ul{R}$ such that the above inequality holds amounts to finding $C > 0$ such that
\be\label{issub}
\int_{L-t}^{L}\,g(s)\,ds \leq C\,g^{2}(L-t)\,,
\quad
\forall 0 \leq t \leq L\,.
\ee
Define
$$
\eta(t) = C\,g^{2}(L-t) - \int_{L-t}^{L}\,g(s)\,ds\,,
$$
so that \rife{issub} is equivalent to proving that $\eta(t) \geq 0$ for every $t$ in $[0,L]$. We have $\eta(0) = 0$, and
$$
\eta'(t) = -2Cg(L-t)g'(L-t) - g(L-t) = g(L-t)[-2Cg'(L-t) - 1]\,,
$$
so that $\eta'(t) > 0$ for every $t$ (which then implies $\eta(t) \geq 0$ as desired) if there exists $C > 0$ such that
\be\label{daverif}
g'(L-t) \leq -\frac{1}{2C}\,,
\quad
\forall 0 \leq t \leq L\,.
\ee
Indeed,  by \rife{derig} we have
 $$
 g'(L-t) = -h(\psi^{-1}(L-t))\,.
 $$ 
 Now, if   $h(0) > 0$ then such a $C$ exists using \rife{h2}, and since $h$ is strictly increasing. On the contrary, if  $h(0) = 0$, \rife{daverif} fails near~0. In this case, however, we have
$$
\eta(L) = Cg^{2}(0) - \int_{0}^{L}\,g(s)\,ds = C - \int_{0}^{L}\,g(s)\,ds\,,
$$
so that there exists $C_{1} > 0$ such that $\eta(L) > 0$. Since $\eta$ is continuous, there exists $\de > 0$ such that $\eta(s) > 0$ for every $s$ in $[L-\de,L]$. Since $h$ is strictly increasing, there exists $C_{2} > 0$ such that \rife{daverif} holds true in $[0,L-\de]$. Taking $C = \max(C_{1},C_{2})$ we then have $\eta(t) \geq 0$ for every $t$, as desired.

Thus, if we define $\ol{R} = R_{L} + \ul{R}$, we have found a function $w$ such that $w(\ol{R}) = 0$, and such that $-\Delta w \leq 2\,g(w)$ in $B_{\ol{R}}(0)$. Let now $R > 0$ be fixed, and consider 
$\ul{v}(x) = w(\ol{R}\,x/R)$. It is easy to see that
$$
-\Delta \ul{v} = -\frac{\ol{R}^{2}}{R^{2}}\Delta w \leq \frac{2\ol{R}^{2}}{R^{2}}g(w) = \frac{2\ol{R}^{2}}{R^{2}}g(\ul{v})\,,
\quad
\ul{v}(|x| = R) = 0\,,
$$
so that $\ul{v}$ is a subsolution of \rife{slin} for $\la = \frac{2\ol{R}^{2}}{R^{2}}$.
\end{proof}

Once we know that, for $\la$ large enough, there exists a radial ``flat'' subsolution in $B_{R}$, we can give the proof of the nonexistence theorem.

\begin{proof}[{\sl Proof of Theorem \ref{nhnl1}}]
Let $x_{0}$ in $\Omega$ and $R > 0$ be such that $B_{R}(x_{0}) \subset \Omega$, and let $\la$ be large enough so that there exists a subsolution $\ul{v}$ of
$$
\bc
-\Delta v = \la\,\rho\,g(v) & \mbox{in $B_{R}(x_{0})$,} \\
\hfill v = 0 \hfill & \mbox{on $\partial B_{R}(x_{0})$,}
\ec
$$
where $\rho$ is given by \rife{assumpt2}, such that $\mis(\{\ul{v} = L\}) > 0$. Such a $\la$ exists by Theorem \ref{sotto}. Define $w$ as $\ul{v}$ in $B_{R}(x_{0})$ and zero in $\Omega \setminus \ol{B_{R}(x_{0})}$. Then $w$ is a subsolution of
$$
\bc
-\Delta v = \la\,\rho\,g(v) & \mbox{in $\Omega$,} \\
\hfill v = 0 \hfill & \mbox{on $\partial\Omega$,}
\ec
$$
hence, by \rife{assumpt2}, a subsolution of \rife{slin}. If $v$ is the solution of \rife{slin}, we have by Lemma \ref{subsup} that $w \leq v \leq L$. Since $\mis(\{w = L\}) > 0$, we have that $\mis(\{v = L\}) > 0$, and so $\La_{f} < +\infty$.
\end{proof}

\section{Limit equation for approximating problems}

In this section we want to describe the behavior of the approximating sequences of solutions to problem \rife{mainlan} for those values of $\la$ such that a solution does not exist. To fix the ideas, we will work with the Laplace operator, under assumption \rife{assumpt2}, with $h(0) = 0$, $f \equiv 1$, and $\Omega = B_{R}(0)$ a ball with radius $R>0$. As in Section~2, let $\{u_{n}\}$ be the sequence of solutions in $ H^1_0 (B_R) \cap L^{\infty}(B_R)$ of 
\be\label{mainappe}
\bc
-\Delta u_n + h_{n}(u_{n})|\D u_n|^2= \la \quad & \mbox{in $B_R$,}\\
\hfill u_n=0 \hfill & \mbox{on $\partial B_R$,}
\ec
\ee
where
$$
h_{n}(s) = 
\bc
h(s) & \mbox{if $s < \si$ and $h(s) \leq n$,}\\
\hfill n \hfill & \mbox{if $s < \si$ and $h(n) > n$, or if $s \geq \si$.}
\ec
$$
Since $h$ is strictly increasing, if $\si_{n}$ is such that $h(\si_{n}) = n$, we have
$$
h_{n}(s) = 
\bc
h(s) & \mbox{if $s \leq \si_{n}$,}\\
\hfill n \hfill & \mbox{if $s > \si_{n}$,}
\ec
$$
with $\si_{n}$ increasing to $\si$ as $n$ tends to infinity. Reasoning as in Section~4, if we define
$$
H_{n}(s) = \int_{0}^{s}\,h_{n}(t)\,dt\,,
\quad 
\psi_{n}(s) = \int_{0}^{s}\,{\rm e}^{-H_{n}(t)}\,dt\,,
\quad
g_{n}(s) = {\rm e}^{-H_{n}(\psi_{n}^{-1}(s))}\,,
$$
then the function $v_{n} = \psi_{n}(u_{n})$ is a solution in $ H^1_0 (B_R) \cap L^{\infty}(B_R)$ of
$$
\bc
-\Delta v_{n} = \la\,g_{n}(v_{n}) & \mbox{in $B_R$,}\\
\hfill v_{n} = 0 \hfill & \mbox{on $\partial B_R$.}
\ec
$$
An explicit calculation yields
$$
g_{n}(s) = 
\bc
\hfill g(s) \hfill & \mbox{if $s \leq \psi(\si_{n})$,} \\
g(\psi(\si_{n})) - n(s - \psi(\si_{n})) & \mbox{if $\psi(\si_{n}) < s \leq \psi(\si_{n}) + \frac{g(\psi(\si_{n}))}{n}$,} \\
\hfill 0 \hfill & \mbox{if $s > \psi(\si_{n}) + \frac{g(\psi(\si_{n}))}{n}$,}
\ec
$$
where $g(s)$ has been defined in Section~4. Since $g$ is concave, and $\si_{n}$ is an increasing sequence of real numbers, $g_{n}(s) \geq g_{n+1}(s) \geq g(s)$, so that (thanks to Lemma \ref{subsup}), $v_{n} \geq v_{n+1} \geq v$, where $v$ is the solution of \rife{slin} given by Theorem \ref{exisslin}.

It is easy to see (using the boundedness of $g_{n}$) that $\{v_{n}\}$ is bounded in $ H^1_0 (B_R)$ so that it converges (the whole sequence, since it is decreasing) to some function $w$. Since $w \leq v_{n} \leq \psi(\si_{n}) + \frac{g(\psi(\si_{n}))}{n}$, and $\si_{n}$ tends to $\si$, we have $0 \leq w \leq \psi(\si) = L$. Furthermore, if $w(x) < L$, then $g_{n}(v_{n}(x))$ converges to $g(w(x))$, while if $w(x) = L$, then $v_{n}(x) \geq L$ for every $n$ in $\na$, and so $0 \leq g_{n}(v_{n}(x)) \leq g(\si_{n})$; therefore, $g_{n}(v_{n}(x))$ converges to $g(w(x)) = 0$. In other words, $w$ is a solution of \rife{slin}, hence {\sl the} solution of the same problem.

We have therefore proved that if $u_{n}$ is the sequence of solutions of \rife{mainappe}, then $v_{n} = \psi_{n}(u_{n})$ converges to the solution of \rife{slin}. If we suppose to be under the assumptions of Theorem \ref{nhnl1}, the solutions of \rife{slin} have a ``flat'' nonzero measure zone $\ol\omega = \{v = L\}$ if $\la$ is large enough; furthermore, since $\Omega$ is a ball, and $v$ is radially symmetric, then its interior $\omega = B_{r}(0)$ is a ball as well (of radius $r$ for some $0 < r < R$). What happens in this case to $u_{n}$? Using the fact that $v_{n}$ decreases to $v$, it is easy to see that $u_{n}$ converges to some function $u$ which has $\ol\omega$ as the ``flat'' zone $\{u = \si\}$.
\medskip

What can we say about the equation? 
Reasoning as in Proposition \ref{prop}, we have that $\{u_{n}\}$ is bounded in $H^1_0 (B_R)$, that the lower order quasilinear term $h_{n}(u_{n})|\D u_{n}|^{2}$ is bounded in $L^1(B_R)$, that $\D u_{n}$ converges almost everywhere to $\D u$, and that $T_{k}(u_{n})$ strongly converges to $T_{k}(u)$ in $H^1_0 (B_R)$ for every $0 < k < \si$.

Since $h_{n}(u_{n})|\D u_{n}|^{2}$ is bounded in $L^1(B_R)$, it converges weakly$^{*}$ in the sense of measures to some nonnegative bounded Radon measure $\nu$. Standard elliptic results then imply that $u$ is the solution of
$$
\bc
-\Delta u + \nu = \la & \mbox{in $B_R$,} \\
\hfill u = 0 \hfill & \mbox{on $\partial B_R$.}
\ec
$$
In addition, since $T_{k}(u_{n})$ strongly converges to $T_{k}(u)$
for every $0 < k < \si$, we have that
$$
h_{n}(u_{n})|\D u_{n}|^{2}\,\chi_{\{u_{n} < \si - \de\}}
\ \mbox{converges to}\ 
h(u)|\D u|^{2}\,\chi_{\{u < \si - \de\}}
\ \mbox{in $L^1(B_R)$,}
$$
for every $\de > 0$, so that
$$
\nu \LL \{u < \si - \de\} 
=
h(u)|\D u|^{2}\,\chi_{\{u < \si - \de\}}\,,
$$
for every $\de > 0$, which implies
$$
\nu \LL \{u < \si\} 
=
h(u)|\D u|^{2}\,\chi_{\{u < \si\}}\,.
$$
Therefore, the limit equation is
$$
\bc
-\Delta u + h(u)|\D u|^{2}\,\chi_{\{u < \si\}} = \la - \nu \LL \{u = \si\} & \mbox{in $B_R$,} \\
\hfill u = 0 \hfill & \mbox{on $\partial B_R$.}
\ec
$$
Observing that $-\Delta u = 0$ in $B_r$, we have 
$$
0 =
\int_{B_r}\,\la\,\psi - \int_{B_r}\,\psi\,d\nu\,,
\quad
\forall \psi \in C^{1}_0(B_r)\,,
$$
which implies $\nu \LL B_r = \la$, so that $u$ solves the problem
\be\label{connu}
\bc
-\Delta u + h(u)|\D u|^{2}\,\chi_{\{u < \si\}} = \la\,\chi_{\{u < \si\}} - \nu \LL \partial\{u = \si\} & \mbox{in $B_R$,} \\
\hfill u = 0 \hfill & \mbox{on $\partial B_R$,} \\
\hfill u = \si \hfill & \mbox{in $B_r$,}
\ec
\ee
which can be rewritten as
$$
\bc
-\Delta u + h(u)|\D u|^{2} = \la & \mbox{in $B_R\backslash\overline{B_r}$,} \\
\hfill u = 0 \hfill & \mbox{on $\partial B_R$,} \\
\hfill u = \si \hfill & \mbox{on $\partial B_r$.} \\
\ec
$$
Hence, the limit function $u$ is a solution of the equation in the set $\{u < \si\}$. Taking a function $\vp\in C^{1}_{0}(B_R)$ for testing the last equation, and integrating on $B_R \setminus \ol{B_r}$, we obtain
$$
-\int_{\partial B_r}\,\D u \cdot \mathbf{n}\,\vp
+
\int_{B_R \setminus \ol{B_r}}\,\D u \cdot \D\vp
+
\int_{B_R \setminus \ol{B_r}}\,h(u)\,|\D u|^{2}\,\vp
=
\la\,\int_{B_R \setminus \ol{B_r}}\,\vp\,,
$$
where $\mathbf{n}$ is the inward normal to the set $B_r$. 
Since $\D u \equiv 0$ on $B_r$, we can rewrite the above identity as
$$
-\int_{\partial B_r}\!\!\!\D u \cdot \mathbf{n}\,\vp
+
\int_{B_R}\!\!\!\D u \cdot \D\vp
+
\int_{B_R}\!\!\!h(u)\,|\D u|^{2}\,\chi_{B_R \setminus \ol{B_r}}\,\vp
=
\la\,\int_{B_R}\!\!\!\chi_{B_R \setminus \ol{B_r}}\,\vp\,,
$$
and comparing with the weak formulation for \rife{connu}, we obtain
$$
\nu \LL \partial B_r = \D u \cdot \mathbf{n}\,,
$$
where $\D u\cdot \mathbf{n}$ is the normal derivative of $u$ ``coming from inside the set where $u < \si$''. We claim that $\D u\cdot \mathbf{n} = 0$; in other words, since $u$ is radially symmetric, we claim that $u'(r+)$, the right derivative of $u$ at $r$, is zero. Indeed, since we have
$$
\int_{B_R \setminus \ol{B_r}}\,h(u)\,|\D u|^{2} < +\infty\,,
$$
and since $u$ is radially symmetric, then
$$
\int_{r}^{R}\,h(u(s))\,|u'(s)|^{2}\,s^{N-1}\,ds < +\infty\,.
$$
Since, as $s$ is close to $r$ we have $u(s) = \si + u'(r+)\,(s - r) + \mbox{o}(s-r)$, then
$$
\int_{r}\,h(u(s))\,|u'(s)|^{2}\,s^{N-1}\,ds
\approx
\int_{r}\,h(\si + u'(r+)\,(s - r))\,|u'(r+)|^{2}\,s^{N-1}\,ds\,.
$$
If $u'(r+) \neq 0$, we then have a contradiction, since
$$
\int_{r}\,h(u(s))\,|u'(s)|^{2}\,s^{N-1}\,ds
\approx
\int_{r}\,h(\si + u'(r+)\,(s - r))\,ds\,,
$$
and, changing variable,
$$
\int_{r}\,h(u(s))\,|u'(s)|^{2}\,s^{N-1}\,ds
\approx
-u'(r+)\int^{\si}\,h(t)\,dt = +\infty\,,
$$
since $h$ satisfies assumption \rife{assumpt2}.

\medskip

In the general case of $\Omega$ a bounded open subset of $\rn$, we conjecture that, under some regularity assumptions on the boundary of the interior $\omega$ of the set $\ol{\omega}=\{u = \si\}$, we have the same result: i.e., the normal derivative of $u$ ``coming from inside the set where $u < \si$'' is zero as in the radial case, and $u$ solves 
$$
\bc
-\Delta u + h(u)|\D u|^{2} = \la & \mbox{in $\Omega\backslash\overline{\omega}$,} \\
\hfill u = 0 \hfill & \mbox{on $\partial \Omega$,} \\
\hfill u = \si \hfill & \mbox{on $\partial \omega$.} \\
\ec
$$

\section{The case $\Lambda_{f} < +\infty$}

In this section, we will suppose that $h$ is such that $\Lambda_{f} < +\infty$; for the sake of simplicity, we will also suppose that $h(0) = 0$. On the function $f$, we will suppose that it is smooth, and that there exists $\rho > 0$ such that $f(x) \geq \rho$ in $\Omega$. Finally, we will deal with the laplacian operator in a smooth domain $\Omega$. We will study the behaviour of the solutions $u_{\la}$ of
\begin{equation}\label{pila}
\begin{cases}
-\Delta u_{\la} + h(u_{\la})|\nabla u_{\la}|^{2} = \la\,f & \mbox{in $\Omega$,} \\
\hfill u_{\la} = 0 \hfill & \mbox{on $\partial\Omega$,}
\end{cases}
\end{equation}
as $\la$ tends to $\Lambda_{f}$.

We recall that, by Theorem \ref{t2+}, $u_{\la}$ is unique and that $\la < \mu$ implies $u_{\la} \leq u_{\mu}$. Let now $\la$ increase to $\Lambda_{f}$, so that $u_{\la}$ increases to some function $u$. It is easy to see that i)-v) of Proposition \ref{prop} hold true, so that (for example), $u_{\la}$ is bounded in $\huz$, and the lower order term $h(u_{\la})|\nabla u_{\la}|^{2}$ is bounded in $\elle1$. Furthermore, since every $u_{\la}$ is strictly smaller than $\si$, we have $u \leq \si$ almost everywhere in $\Omega$. Is it true that the limit function $u$ is a solution for $\la = \Lambda_{f}$? If the set $\{u = \si\}$ has positive measure, then clearly this is not true, and the considerations of Section~5 apply. If, however, the measure of the set $\{u = \si\}$ is zero, vi) of Proposition \ref{prop} holds true, and for every $\de > 0$ there exists $\tau > 0$ such that
$$
{\rm meas}(\{\si-\tau \leq u \leq \si+\tau\}) \leq \de\,.
$$
Since $u_{\la}$ increases to $u$, this means that for every $\de > 0$ there exists $\tau > 0$ such that
$$
{\rm meas}(\{\si-\tau \leq u_{\la} \leq \si+\tau\}) \leq \de\,, \quad \forall \la < \Lambda_{f}\,.
$$
In other words, the assumptions of Lemma \ref{cpt} are satisfied, and so $h(u_{\la})|\nabla u_{\la}|^{2}$ is compact in $\elle1$. Therefore, we can pass to the limit in the equation satisfied by $u_{\la}$, to have that $u$ is a solution for $\la = \Lambda_{f}$.

What we do not know is whether in the general case the measure of the set $\{u = \si\}$ is zero or not; i.e., if there exists a solution for $\la = \Lambda_{f}$. What we can prove is however that the set $\{u = \si\}$ (where $u$ is as above the limit of $u_{\la}$ as $\la$ increases to $\Lambda_{f}$) cannot be empty. Indeed, suppose that $\{u = \si\} = \emptyset$. Since both $f$ and $\Omega$ are smooth, standard elliptic results imply that $\{u_{\la}\}$ is equi-H\"{o}lder continuous: one can see that this is true by performing the (lipschitz continuous) change of variable as in Section~4 and use De Giorgi's theorem to prove that the solutions of the semilinear equation are equi-H\"{o}lder continuous. Thus, the convergence of $u_{\la}$ to $u$ is uniform, and so $u$ is continuous. This implies that there exists $\eps > 0$ such that $0 \leq u \leq \si-\eps$ in $\Omega$. We claim that if this is the case, than there exist solutions strictly smaller than $\si$ for some $\la > \Lambda_{f}$ (and this contradicts the definition of $\Lambda_{f}$). Indeed, define $\tilde{h}$ as follows:
$$
\tilde{h}(s) = \begin{cases}
\hfill h(s) \hfill & \mbox{if $0 \leq s \leq \si - \frac\eps2$,} \\
\hfill h(\si - \frac{\eps}{2}) \hfill & \mbox{if $s > \si - \frac{\eps}{2}$.}
\end{cases}
$$
Then clearly $u$ solves
$$
-\Delta u + \tilde{h}(u)|\nabla u|^{2} = \Lambda_{f}\,f\,,
$$
and, by the results of \cite{bbm}, there exists a solution $v$ of
$$
-\Delta v + \tilde{h}(v)|\nabla v|^{2} = (\Lambda_{f}+\de)\,f\,,
$$
with $v \geq u$, for every $\de > 0$. It is easy to see, reasoning as in the proof of Theorem \ref{t2+} (or performing the change of variable as in Section~4), that
$$
\norma{u-v}{\elle\infty} \leq C\,\de\,\norma{f}{\elle\infty}\,,
$$
so that, if $\de$ is small enough, $0 \leq v \leq \si - \frac{3\eps}{4}$, and so $v$ is a solution of
$$
-\Delta v + h(v)|\nabla v|^{2} = (\Lambda_{f}+\de)\,f\,,
$$
thus contradicting the definition of $\Lambda_{f}$.

Therefore, we have that, as $\la$ tends to $\Lambda_{f}$, $u_{\la}$ increases to some function $u$ which ``touches'' $\si$. If it touches it on a set of positive measure, then $u$ is not a solution, while it is a solution if the measure of $\{u = \si\}$ is zero. In the first case, it is clear that no solution exists for $\la > \Lambda_{f}$, but what happens if we have a solution for $\la = \Lambda_{f}$? Do solutions exist for $\la > \Lambda_{f}$? We do not know what happens in the general case, but we can give an answer in the radial case: there can be only one solution $u$ such that the set $\{u = \si\}$ is not empty and has zero measure.

To simplify the calculations, rewrite the equation as
\begin{equation}\label{gisotto}
\begin{cases}
-\Delta u + \frac{|\nabla u|^{2}}{g(u)} = \la\,f & \mbox{in $\Omega$,}\\
\hfill u = 0 \hfill & \mbox{on $\partial\Omega$,}
\end{cases}
\end{equation}
where $\Omega = B_{R}(0)$ is a ball, and $f$ is a smooth radial function. On the function $g$ we will suppose that it is a positive, $C^{1}((0,\si))$ function such that
$$
\lim_{s \to \si^{-}}\,g(s) = 0\,,
\quad
\exists \lim_{s \to \si^{-}}\,g'(s) \in [-\infty, 0]\,,
\quad
\int_{0}^{\si}\,\frac{dt}{\sqrt{g(t)}} < +\infty\,,
$$
for some $\si > 0$. The function $g(s) = (1-s)^{\ga}$ obviously satisfies the above assumption if $\ga < 2$ and $\si = 1$. We are going to prove that if $\Lambda$ is such that there exists a $C^{2}$ radial solution $u_{\Lambda}$ such that $u_{\Lambda}(0) = \si$, and $u_{\Lambda}(r) < \si$ for every $r > 0$, then there exist solutions strictly smaller than $\si$ if $\la < \Lambda$ (i.e., $\Lambda_{f} = \Lambda$), and there exists no solution for $\la > \Lambda$.

Since $u_{\Lambda}$ is a radial solution, we have
$$
-u''_{\Lambda}(r) - (N-1)\frac{u'_{\Lambda}(r)}{r} + \frac{(u'_{\Lambda}(r))^{2}}{g(u_{\Lambda}(r))} = \Lambda\, f(r)\,, \quad \forall r > 0\,,
$$
and so, by the de l'H\^opital rule, which can be applied since $u'_{\Lambda}(0) = 0$,
$$
\lim_{r \to 0^{+}}\,\frac{(u'_{\Lambda}(r))^{2}}{g(u_{\Lambda}(r))} = \Lambda\,f(0) + Nu''_{\Lambda}(0)\,.
$$
We know, again by the de l'H\^opital rule, that if it exists
\begin{equation}\label{limsec}
\lim_{r \to 0^{+}}\,\frac{2u'_{\Lambda}(r)\,u''_{\Lambda}(r)}{g'(u_{\Lambda}(r))\,u'_{\Lambda}(r)}
= \lim_{r \to 0^{+}}\,\frac{2u''_{\Lambda}(r)}{g'(u_{\Lambda}(r))}\,,
\end{equation}
then this limit is equal to
$$
\lim_{r \to 0^{+}}\,\frac{(u'_{\Lambda}(r))^{2}}{g(u_{\Lambda}(r))}\,.
$$
Since $u_{\Lambda}$ is a $C^{2}$ function, and the origin is a maximum for $u_{\Lambda}$, we have
$$
\lim_{r \to 0^{+}}\,u''_{\Lambda}(r) = u''_{\Lambda}(0) \leq 0\,. 
$$
Then we may have three possibilities.

\begin{itemize}
\item[i)] if $g'(\si) = -\infty$, then limit \rife{limsec} is equal to zero, and $u''_{\Lambda}(0) = -\frac{\Lambda\,f(0)}{N}$;

\item[ii)] if $g'(\si) < 0$, then limit \rife{limsec} is equal to $\frac{2u''_{\Lambda}(0)}{g'(\si)}$, and $u''_{\Lambda}(0) = -\frac{\Lambda\,f(0)}{N - \frac{2}{g'(\si)}}$;

\item[iii)] if $g'(\si) = 0$ then the limit \rife{limsec} is not finite if $u''_{\Lambda}(0) < 0$, which will yield that $\Lambda\,f(0) + N\,u''_{\Lambda}(0) = -\infty$, a contradiction; therefore, if $g'(\si) = 0$, we have $u''_{\Lambda}(0) = 0$.
\end{itemize}
If we are in cases i) or ii), we have
$$
u_{\Lambda}(r) = \si - C\Lambda\,r^{2} + {\rm o}(r^{2})\,,
$$
where $C$ is either $\frac{f(0)}{N}$ or $\frac{f(0)}{N - \frac{2}{g'(\si)}}$. Let now $\la < \Lambda$, and let $u_{\la}$ be a solution of our problem: by Theorem \ref{t2+} we have that $u_{\la} \leq u_{\Lambda}$. If $u_{\la}(0) < \si$, we have nothing to prove. If $u_{\la}(0) = \si$, then we can repeat the above proof to have  
$$
u_{\la}(r) = \si - C\,\la\,r^{2} + {\rm o}(r^{2})\,,
$$
with $C$ as above. Since $\la < \Lambda$, this contradicts the fact that $u_{\la} \leq u_{\Lambda}$, and so $u_{\la}(0)$ has to be smaller than $\si$. Analogously, we arrive to a contradiction if we suppose that there exists a solution $u_{\la}$ with no ``flat'' zones for $\la > \Lambda$. Thus, in both cases i) and ii) we have that $\Lambda_{f} = \Lambda$ is a supremum, and there is no solution for $\la > \Lambda_{f}$.

Now we turn our attention to case iii), which is more delicate. In this case, since $u''_{\Lambda}(0) = 0$, we have
\be\label{vucinic}
\lim_{r \to 0^{+}}\,\frac{(u'_{\Lambda}(r))^{2}}{g(u_{\Lambda}(r))} = \Lambda\,f(0) > 0\,.
\ee
Define now, for $s \leq \si$, the bounded function
$$
G(s) = \int_{0}^{s}\,\frac{dt}{\sqrt{g(t)}}\,,
$$
and observe that \rife{vucinic} can be rewritten as
$$
\lim_{r \to 0^{+}}\,\left| \frac{d}{dr}\,G(u_{\Lambda}(r))\right|^{2} = \Lambda\,f(0)\,.
$$
Thus, if we define the function $v_{\Lambda}(r) = G(u_{\Lambda}(r))$, then $v_{\Lambda}$ is a decreasing function (since $u_{\Lambda}$ is decreasing) such that $v'_{\Lambda}(0) = -\sqrt{\Lambda\,f(0)}$. Thus,
$$
v_{\Lambda}(r) = G(\si) -\sqrt{\Lambda\,f(0)}\,r + {\rm o}(r)\,.
$$
Let now $\la < \Lambda$. If the solution $u_{\la}$ is such that $u_{\la}(0) = \si$, we can repeat the above proof to have that $v_{\la}(r) = G(u_{\la}(r))$ satisfies
$$
v_{\la}(r) = G(\si) -\sqrt{\la\,f(0)}\,r + {\rm o}(r)\,,
$$
so that $v_{\la}(r) \geq v_{\Lambda}(r)$ in a neighbourhood of~0. This yields that $u_{\la}(r) \geq u_{\Lambda}(r)$ in a neighbourhood of~0, and so $u_{\la}(r) \equiv u_{\Lambda}(r)$ near the origin, a contradiction since $\la \neq \Lambda$. Thus, it has to be $u_{\la}(0) < \si$, as desired. If $\la > \Lambda$, the same argument shows that there is no solution $u_{\la}$ with no ``flat'' zones $\{u_{\la} = \si\}$.

\end{document}